\documentclass[journal]{IEEEtran}

\usepackage{mathptmx}
\usepackage{amsthm,amsmath,amssymb,amsfonts}
\usepackage{bm} 
\usepackage{subeqnarray}
\usepackage{cases,cite}
\usepackage[utf8]{inputenc}
\usepackage{float}
\usepackage{graphicx} 
\usepackage{subfigure}
\usepackage{url}
\usepackage{multirow}
\usepackage{microtype}
\usepackage{booktabs}
\usepackage{color}
\usepackage[table,xcdraw]{xcolor}

\newtheorem{proposition}{Proposition}

\newtheorem{theorem}{Theorem}
\newtheorem{lemma}{Lemma}

\newtheorem{assumption}{Assumption}
\newtheorem{remark}{Remark}
\newtheorem{condition}{Condition}

\IEEEoverridecommandlockouts

\begin{document}

\title{Consensus-based Distributed Discrete Optimal Transport for Decentralized Resource Matching}

\author{Rui~Zhang, Quanyan~Zhu

\thanks{R. Zhang and Q. Zhu are with the Department of Electrical and Computer Engineering, New York University, Brooklyn, NY, 11201
E-mail:\{rz885,qz494\}@nyu.edu. }}

\maketitle

\begin{abstract}
Optimal transport has been used extensively in resource matching to promote the efficiency of resources usages by matching sources to targets. However, it requires a significant amount of computations and storage spaces for large-scale problems. In this paper, we take a consensus-based approach to decentralize discrete optimal transport problems and develop fully distributed algorithms with alternating direction method of multipliers. We show that our algorithms guarantee certain levels of efficiency and privacy besides the distributed nature. We further derive primal and dual algorithms by exploring the primal and dual problems of discrete optimal transport with linear utility functions and prove the equivalence between them. We verify the convergence, online adaptability, and the equivalence between the primal algorithm and the dual algorithm with numerical experiments. Our algorithms reflect the bargaining between sources and targets on the amounts and prices of transferred resources and reveal an averaging principle which can be used to regulate resource markets and improve resource efficiency. 
\end{abstract}


\section{Introduction}
Resource matching plays an important role in many applications, such as emergency response\cite{purohit2013emergency, fiedrich2000optimized, willis2007guiding}, smart grids\cite{wijaya2013matching,samadi2010optimal}, wireless networks\cite{georgiadis2006resource, steenstra2014dynamic, gu2015matching}, and data centers\cite{ferris2013methods, ding2014qos}. The objective of matching is to assign resources with targets in a way that promotes resource efficiency and increases social benefits. For example, proper assignments of workers with heterogeneous skills to jobs with heterogeneous characteristics can improve the total economic output\cite{galichon2015optimal}; efficient allocations of incoming customers to different parking slots can reduce the average time and costs of finding parking spaces and improve the overall parking capacity in smart cities\cite{geng2013new}.  

Optimal transport is one of the centralized planning approaches to resource matching\cite{villani2008optimal, galichon2015optimal}. A central planner finds the optimal scheme to transport or move resources from their sources to the targets that maximizes the social welfares. However, the computations required to solve such planning problem grow exponentially with the increase in the numbers of sources and targets in large-scale problems. Moreover, the communication overhead required for the planner to collect information and coordinate between sources and targets is also significant with a large number of participants (See Fig. \ref{fig:Example}(a)). In certain cases, targets and sources may not be willing to share any information to the central planner (See Fig. \ref{fig:Example}(b)). Therefore, it is computationally and structurally costly for the central planner to make a global planning for large-scale systems despite the recent efforts on developing methods to speed up the computations.

\begin{figure}[]
\centering
\subfigure[Power Grids]{\includegraphics[width=0.235\textwidth]{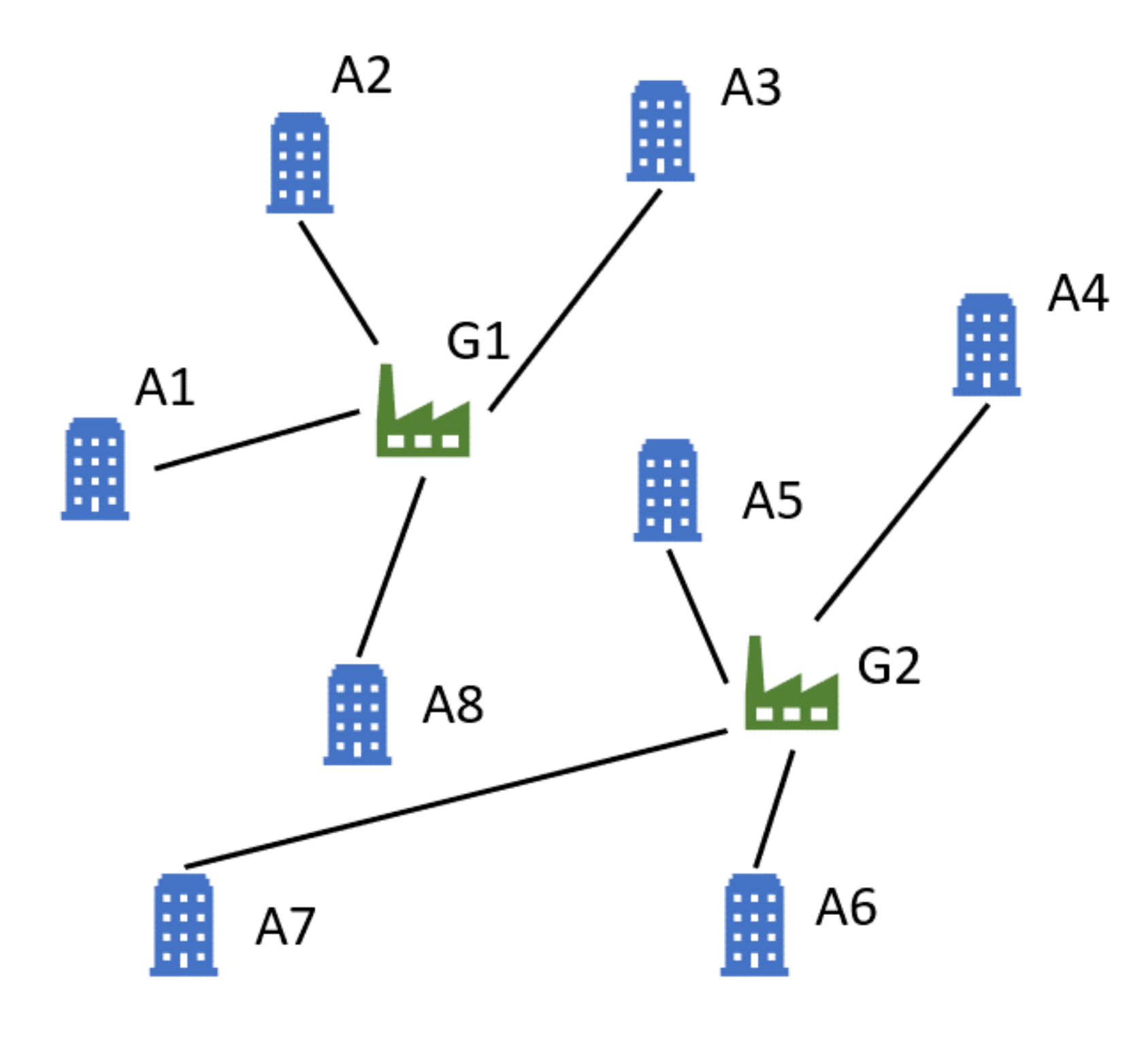}}
\subfigure[Medical Resources]{\includegraphics[width=0.235\textwidth]{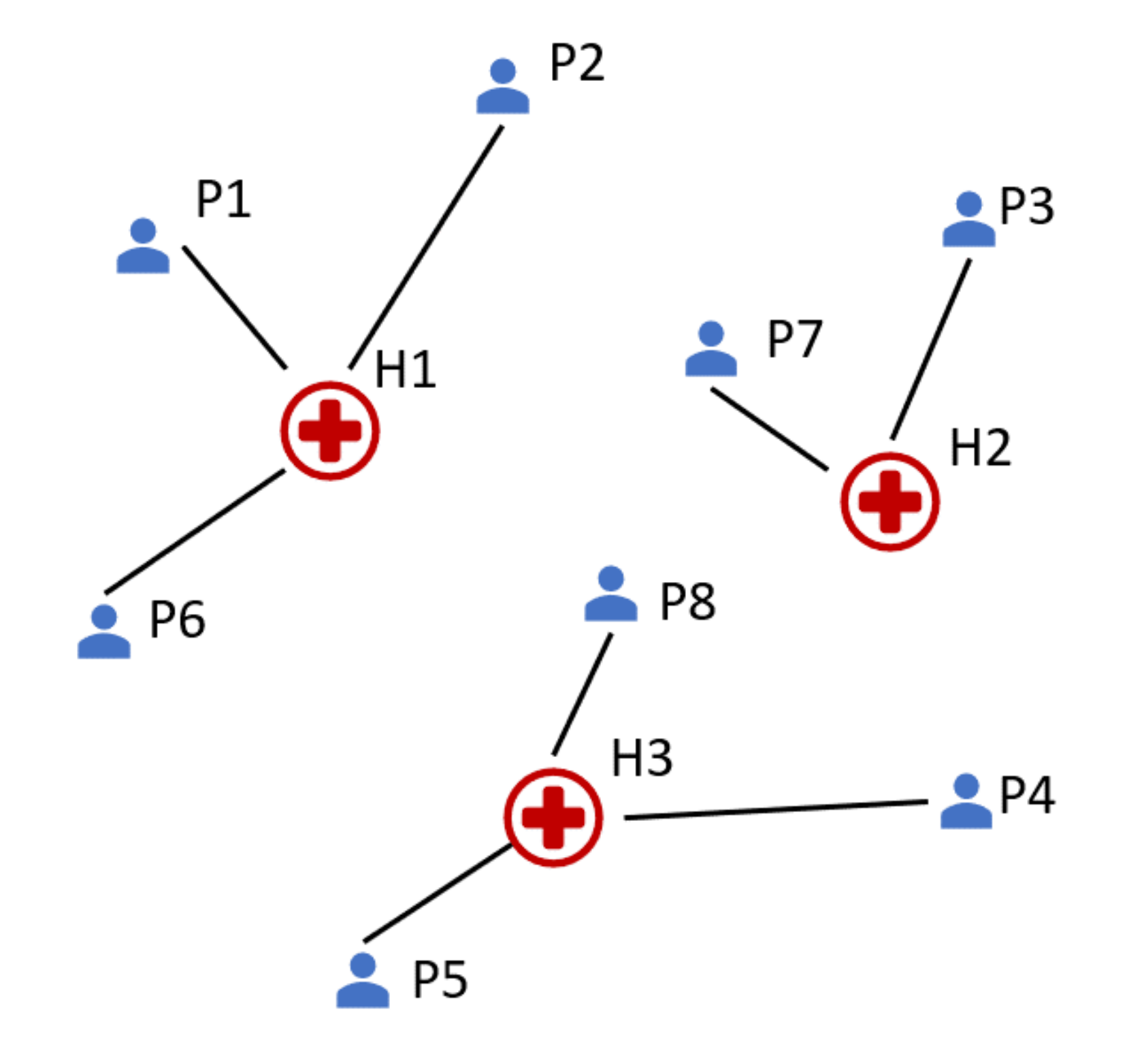}}
\caption{{Examples of decentralized resource matching. Fig. (a) shows a smart grid system whose objective is to increase energy efficiency by matching different customers to different power plants. The centralized matching is not efficient with the increase of the sizes of customers and available power plants as it is challenging to collect the information of them in a short period of time. Fig. (b) illustrates matching patients with hospitals. Each patient has his or her own preferences on hospitals, such as locations, affordability, and proficiency. It is impossible to collect all the information of patients and hospitals and then compute the optimal matching between them in a centralized way as patients may not be willing to share their preferences with others. } }
\label{fig:Example}
\end{figure}

Decentralized methods and algorithms offer practical solutions to reduce the complexity of the computations and make large-scale matching problems feasible \cite{gao2009distributed,hasan2015distributed, raman1998matchmaking}. In this paper, we aim to decentralize discrete optimal transport problems and derive a fully distributed algorithm which dispenses with a central planner to develop the matching mechanism between sources and targets. We consider a consensus-based approach in which consensus constraints are used to capture the agreements between the participants and the central planner on the matching. This approach allows the derivations of a fully distributed algorithm for discrete optimal transport with alternating direction method of multipliers (ADMM)\cite{boyd2011distributed}.

Our algorithm does not require central planners to compute the matching between sources and targets or store their information. Each participant computes its own problem and communicates with only the participants on the other side, which enables parallel computations and yields high efficiency compared to centralized algorithms. The consensus between participants and central planners is simplified into direct consensus between sources and targets. Each pair of source and target updates their matching until they reach an agreement. When all sources and targets reach consensus, the resulting matching achieves the optimal social utility and coincides with the centralized solution.

Besides the property of distributed computation and the improved computational efficiency, our algorithm offers additional useful features on distributed resource matching. Our algorithm guarantees a certain level of privacy as there are no central planners to collect information and each participant only shares information with the participants on the other side. Our algorithm can be easily implemented online and does not require rerunning the algorithm when participants leave or enter the matching platform. Instead, the algorithm can update the matching and adapt to changes in the preferences, resources, and targets.

We further explore the dual problem of discrete optimal transport with linear utilities which naturally provides a way to study pricing schemes of resource matching. We leverage the consensus-based approach to capture the bargaining and agreement process of the participants on the prices of the resources. We use ADMM to derive a dual algorithm which shares similar characteristics with the primal algorithm on fully distributed nature, efficiency, privacy, and online adaptability. We further prove the equivalence between the primal algorithm and the dual algorithm, and show that each sub-problem in the dual algorithm is the dual problem of each sub-problem in the primal algorithm.  

Our algorithms provide useful insights of resource matching in market environments. The optimal matching is achieved through the bargaining between sources and targets on the amounts and prices of transferred resources in the primal algorithm and the dual algorithm, respectively. Each participant proposes its offers with the objective of maximizing its utility while minimizing its differences from the participants on the other side. When all pairs of sources and targets reach consensus on their motions, the achieved matching maximizes the overall surplus. 

Furthermore, our algorithms reveal an averaging principle that sources and targets follow during their bargaining process. The core idea is that a source and a target should propose their current offers to be close to the average of their previous offers. The principle indicates a fair and unbiased negotiation between sources and targets. The convergence of our algorithms to the optimal matching shows that the sources and targets can maximize their total surplus by following the averaging principle. Thus, the  principle can be used to regulate sources and targets in resource markets and improve the resource efficiency. 

Our contributions can be summarized as follows: 
\begin{itemize} 
\item We propose a consensus-based approach to decentralize discrete optimal transport problems and further derive a fully distributed resource matching algorithm using ADMM. Each participant computes its own matching decisions and only communicates with participants on the other side. 
\item We demonstrate that our algorithm guarantees certain levels of efficiency and privacy and show that our algorithm is fully adaptive when changes occur in the preferences, resources, and targets.
\item We derive a fully distributed dual algorithm that provides a pricing scheme of resource matching in distributed markets and reveal an averaging principle that a source and a target always compromise to the average of either amounts or prices of transferred resources between them. 
\item We use numerical experiments to demonstrate the convergence and online capability of our algorithm, and we also corroborate the equivalence of primal algorithm and dual algorithm. 
\end{itemize}

\subsection{Related Works}
Our work is related to the literature on resource allocation, optimal transport, and distributed optimization. Resource allocation deals with the assignment or distribution of available resources in efficient ways \cite{jain1984quantitative, ross1999needs}. Recent applications include natural resource management such as lands and forests \cite{abrahamson1979patterns, bouis1990effects}, sensor networks \cite{gao2009distributed, mainland2005decentralized}, smart grids \cite{farbod2007resource, wijaya2013matching}, human resources \cite{lepak1999human}, and emergency response \cite{fiedrich2000optimized, willis2007guiding}.

Resource matching is one class of resource allocation with multiple sources and multiple targets, and it finds the optimal matching between multiple sources and multiple targets to transfer resources and maximize the total surplus of them\cite{purohit2013emergency, fiedrich2000optimized, willis2007guiding,wijaya2013matching,samadi2010optimal,georgiadis2006resource, steenstra2014dynamic, gu2015matching,ferris2013methods, ding2014qos}. Various approaches have been developed to study resource matching, which can be summarized into two distinct categories. Planning-based approaches focus on the designation of optimal resource matching to maximize social welfare \cite{galichon2015optimal}, while market-based approaches study interactions between sources and targets in markets through the analysis of supply and demand of resources\cite{clearwater1996market}. 

It is common for most planning-based frameworks to have a central planner identify resources, sources, and targets, and then design the matching\cite{galichon2015optimal}. However, this approach requires a significant amount of computation for large-scale problems. A lot of techniques such as approximation \cite{devanur2011near}, belief propagation\cite{huang2011fast}, barrier methods\cite{madan2010fast}, and network simplex\cite{bonneel2011displacement, peyre2017computational}, have been developed to speed up the computations. 

Although market-based approaches do not require central planners to design the matching between sources and targets, they have their own disadvantages\cite{clearwater1996market}. On one hand, it is difficult to study every source and target and further analyze their complex interactions for large-scale problems; on the other hand, it is impractical to derive the optimal resource matching as any matching pair is affected by all the other matching pairs through the interconnections between sources and targets.   

In this paper, we build on the discrete optimal transport framework which is a centralized planning-based approach to resource matching, and aim to derive a fully distributed algorithm to find the optimal matching between sources and targets. Optimal transport studies the optimal allocation of resources and provides economic interpretations and analysis of assignment problems and their properties \cite{geng2013new, galichon2015optimal}. Besides resource matching, it has also been applied to areas such as machine learning \cite{cuturi2013sinkhorn}, image processing\cite{haker2004optimal}, and reflector design\cite{glimm2010rigorous}.

A few distributed methods have been introduced to decentralize resource matching. For example, Gao et al. have developed a distributed algorithm to match resources to events in sensor networks by exploring the tree metric from the underlying network metric \cite{gao2009distributed}; Buyya has proposed distributed computational economy as an effective metaphor for the management of resources and application scheduling in grid computing \cite{buyya2002economic}; Hasan et al. have developed a distributed approach using stable matching to allocate radio resources for device-to-device communication \cite{hasan2015distributed}. 

We consider a consensus-based approach using alternating direction method of multipliers (ADMM) to decentralize discrete optimal transport problems\cite{boyd2011distributed, wei2012distributed}. ADMM provides a natural methodology to develop distributed algorithms for decentralized applications  \cite{forero2010consensus,chang2015multi, huang2016consensus}. For example, Zhang et al. have developed a consensus-based distributed algorithm with ADMM to study the interactions between attackers and machine learning learners in networked environments \cite{zhang2018game}; Zennaro et al. have proposed a consensus-based clock synchronization algorithm with ADMM \cite{zennaro2011fast}; Shen et al. have addressed the distributed robust multicell coordinated beamforming by taking an ADMM-based approach\cite{shen2012distributed}.

ADMM has been applied to solve transportation problems. For example, Papadakis, et al. have shown proximal splitting schemes for solving discretized dynamical optimal transport\cite{papadakis2014optimal}; Benamou, et al. have presented augmented Lagrangian methods to solve the time-dependent optimal transport problems\cite{benamou2015augmented}; Geissler, et al. have proposed an MIP-based alternating direction method to solve power-constrained gas transportation problems\cite{geissler2015solving}. They have all focused on centralized algorithms. In this work, we aim to address decentralized resource matching and leverage ADMM to achieve a fully distributed algorithm that boasts key features for decentralized systems.

\subsection{Organization of the Paper}
The rest of the paper is organized as follows. Section \ref{sec:PF} formulates centralized discrete optimal transport problems. Section \ref{sec:DOT} develops a distributed algorithm for resource matching and further extends it into an online version. Section \ref{sec:DD} explores the primal algorithm and the dual algorithm of discrete optimal transport problems with linear utility functions and proves the equivalence between them. Section \ref{sec:NE} provides numerical experiments and Section \ref{sec:CON} presents concluding remarks. Appendix A summarizes useful results of ADMM. Appendices B, C, D, and E provide the proofs for Proposition \ref{pro:MainReRe}, Theorem \ref{the:OTReDual}, Proposition \ref{pro:OTRe2}, and Proposition \ref{pro:OTReDualSolDual}, respectively. The following table provides a summary of notations in our paper.

\begin{table}[ht]
\begin{tabular}{cc}
\hline
\multicolumn{2}{c}{Summary of Notations}    
\\ \hline \multicolumn{1}{c|}{$x$, $y$} & Target $x$, Source $y$ 
\\ \multicolumn{1}{c|}{$\mathcal{X}$, $\mathcal{X}_y$} & All Targets, Targets Connected to Source $y$\\ \multicolumn{1}{c|}{$\mathcal{Y}$, $\mathcal{Y}_x$} & All Sources, Sources Connected to Target $x$
\\ \multicolumn{1}{c|}{$\mathcal{G}$} & Undirected Bipartite Graph Between $\mathcal{X}$ and $\mathcal{Y}$ 
\\ \multicolumn{1}{c|}{$\pi_{xy}$} & Amount of Resources Transferred from $y$ to $x$ 
\\ \multicolumn{1}{c|}{${\Pi_x}$, ${\Pi_y}$, ${\Pi_{\mathcal{G}}}$} & Amounts of Resources for $x$, $y$, $\mathcal{G}$
\\ \multicolumn{1}{c|}{${\Pi_x^{(t)}}$, ${\Pi_{\mathcal{G}}^{(t)}}$} &  Amounts of Resources Proposed by $x$, $\mathcal{X}$
\\ \multicolumn{1}{c|}{${\Pi_y^{(s)}}$, ${\Pi_{\mathcal{G}}^{(s)}}$} & Amounts of Resources Proposed by $y$, $\mathcal{Y}$
\\ \multicolumn{1}{c|}{$u_x$, $v_y$} & Surplus of $x$, Surplus of $y$
\\ \multicolumn{1}{c|}{${\mathbf{u}_\mathcal{X}}$, ${\mathbf{v}_\mathcal{Y}}$} & Surpluses of $\mathcal{X}$, Surpluses of $\mathcal{Y}$
\\ \multicolumn{1}{c|}{$w_{xy}$} & Price of Resources Transferred from $y$ to $x$ 
\\ \multicolumn{1}{c|}{${\mathbf{w}_x}$, ${\mathbf{w}_y}$, ${\mathbf{w}_\mathcal{G}}$} & Prices of Resources for $x$, $y$, $\mathcal{G}$
\\ \multicolumn{1}{c|}{${\mathbf{w}_x^{(t)}}$, ${\mathbf{w}_{\mathcal{G}}^{(t)}}$} & Prices of Resources Proposed by $x$, $\mathcal{X}$
\\ \multicolumn{1}{c|}{${\mathbf{w}_y^{(s)}}$, ${\mathbf{w}_{\mathcal{G}}^{(s)}}$} & Prices of Resources Proposed by $y$, $\mathcal{Y}$
\\ \hline
\end{tabular}
\end{table}

\section{Problem Formulation}
\label{sec:PF}
In this section, we present a discrete optimal transport framework for resource matching between targets and sources. Let $\mathcal{X}:=\{1, ..., N\}$ and $\mathcal{Y}:=\{1, ..., M\}$ denote the sets of targets and sources, respectively. We further use $\mathcal{X}_y\subseteq \mathcal{X}$ and $\mathcal{Y}_x\subseteq \mathcal{Y}$ to denote the subset of targets connected to source $y$ and the subset of sources connected to target $x$, respectively. The subsets $\mathcal{X}_y$ and $\mathcal{Y}_x$ capture the situation where each participant has only limited choices to match with the participants on the other side, which often exists in large-scale matching problems. 

We can represent the possible matching between targets and sources by an undirected bipartite graph $\mathcal{G}:= \left\lbrace    \{ x,y \} \left|  y\in\mathcal{Y}_x, x \in \mathcal{X}       \right. \right\rbrace = \left\lbrace    \{ x,y \} \left|  x\in\mathcal{X}_y,  y \in \mathcal{Y}       \right. \right\rbrace$. ``Undirected" indicates that if a target $x$ knows a source $y$, i.e., $y\in\mathcal{Y}_x$, then the source $y$ must know the target $x$, i.e, $x\in\mathcal{X}_y$. ``Bipartite" indicates that the only connections in the graph are between targets $\mathcal{X}$ and sources $\mathcal{Y}$ to transfer resources. 

Let $\pi_{xy}\in\mathbb{R}_{\geq 0}$ denote the amount of resources that source $y$ provides to target $x$, and ${\Pi_\mathcal{G}}:=\{\pi_{xy}|\{x, y\}\in\mathcal{G}\}$ denote the set of the amounts of resources from sources $\mathcal{Y}$ to targets $\mathcal{X}$ through the graph $\mathcal{G}$. We further use ${\Pi_{\overline{x}}}:=\{\pi_{xy}|y\in\mathcal{Y}_x, x=\overline{x}\}$ and ${\Pi_{\overline{y}}}:=\{\pi_{xy}|x\in\mathcal{X}_y, y = \overline{y}\}$ to represent the sets of the amounts of resources to target $\overline{x}$ and from source $\overline{y}$, respectively. It is easy to achieve that $\Pi_{\overline{x}} \subseteq \Pi_\mathcal{G}$ and  $\Pi_{\overline{y}} \subseteq \Pi_\mathcal{G}$.

The objective of discrete optimal transport is to find the optimal matching $\Pi_\mathcal{G}^*$ that maximizes the total surplus of targets and sources, which can be achieved by solving the following optimization problem:  
\begin{equation}
\label{eq:Main}
\begin{array}{c}
\max\limits_{{\Pi_\mathcal{G}}}   \sum\limits_{x\in\mathcal{X}}  \sum\limits_{y\in \mathcal{Y}_x}  f_{xy}(\pi_{xy})+ \sum\limits_{y\in\mathcal{Y}}  \sum\limits_{x\in \mathcal{X}_y}g_{xy} (\pi_{xy}) \\
\begin{array}{cc}
{\begin{array}{c}
	\text{s.t.} \\ \ \\ \ \\ \
	\end{array}}&{\begin{array}{cc}
	{p_{x, l} \leq \sum\limits_{y\in \mathcal{Y}_x}  \pi_{xy}  \leq p_{x, h},}&{\forall x\in\mathcal{X};} \\
	{q_{y, l} \leq \sum\limits_{x\in \mathcal{X}_y}  \pi_{xy} \leq q_{y, h},}&{\forall y\in\mathcal{Y};} \\
	{\pi_{xy}\geq 0,}&{\forall \{x, y\}\in\mathcal{G}.} 
	\end{array}}
\end{array}
\end{array}
\end{equation} 
The objective function captures the total surplus given the matching rules ${\Pi_\mathcal{G}}$. Functions {$f_{xy}:\mathbb{R}_{\geq 0} \rightarrow \mathbb{R}$ and $g_{xy}:\mathbb{R}_{\geq 0} \rightarrow \mathbb{R}$} represent the utility functions for $x$ and $y$, respectively. The constraints capture each participant's boundaries on the amount of transferred resources.  $p_{x, l}$ and $p_{x,h}$ represent the lower and upper bounds for $x$, respectively, and they satisfy $0 \leq p_{x,l} \leq p_{x, h}$; $q_{y, l}$ and $q_{y, h}$ represent the lower and upper bounds for $y$, respectively, and they satisfy $0 \leq q_{y, l} \leq q_{y, h}$. 

However, there may exist no ${\Pi_\mathcal{G}}$ that satisfies the constraints, and in that case, we cannot find the optimal matching by solving problem (\ref{eq:Main}). For example, when all sources $\mathcal{Y}_x$ together cannot provide sufficient resources to meet the minimum requirement of target $x$, i.e., $p_{x, l} > \sum_{y\in\mathcal{Y}_x}q_{y, h}$, no ${\Pi_{x}}$ is feasible for $x$. In this paper, we focus on situations with feasible ${\Pi_\mathcal{G}}$. 
\begin{condition}
\label{con:MainNecessary}
(Necessity) The following inequalities must hold to guarantee the feasibility of ${\Pi_\mathcal{G}}$:
\[(a). \ \ p_{x, l} \leq \sum_{y\in\mathcal{Y}_x}q_{y, h}, \ \ \ \forall x \in \mathcal{X};  \] 
\[(b). \ \ \sum_{x\in\mathcal{X}} p_{x, l} \leq \sum_{y\in\mathcal{Y}}  q_{y, h}.  \] 
\end{condition}
Condition \ref{con:MainNecessary}(a) indicates that any target can acquire the minimum amount of resources. Condition \ref{con:MainNecessary}(b) indicates that the demands are lower than or equal to the supplies. 
\begin{condition}
\label{con:MainSufficient}
(Sufficiency) The feasibility of ${\Pi_\mathcal{G}}$ is guaranteed if the following inequalities hold:
\[ q_{y, h} \geq \sum_{x\in\mathcal{X}_y}p_{x, l}, \ \ \ \forall y \in \mathcal{Y};  \] 
\end{condition}
Condition \ref{con:MainSufficient} indicates that any supplier can provide sufficient resources to satisfy the requirements of all the targets it knows. Condition \ref{con:MainSufficient} is stronger than Condition \ref{con:MainNecessary}. Note that problems that do not meet Condition \ref{con:MainSufficient} but satisfy Condition \ref{con:MainNecessary} can be still feasible. 

We have the following assumption regarding the convexity of the utility functions $f_{xy}$ and $g_{xy}$.
\begin{assumption}
\label{ass:fg}
Both the utility functions $f_{xy}$ and $g_{xy}$ are concave on $\pi_{xy}$, $\forall \{x, y\} \in\mathcal{G}$.
\end{assumption}
We use the following remarks to provide examples and discuss the interpretations behind $f_{xy}$ and $g_{xy}$.
\begin{remark}
\label{rem:fxy}	
The target $x$'s utility function $f_{xy}: \mathbb{R}_{\geq 0}\rightarrow\mathbb{R}$ can be expressed as follows:
\begin{equation}
\label{eq:remfxy}
f_{xy}(\pi_{xy}) = U_{xy}^{(t)}(\pi_{xy}) - C_{xy}^{(t)}(\pi_{xy}), 
\end{equation}
where $U_{xy}^{(t)}: \mathbb{R}_{\geq 0}\rightarrow\mathbb{R}_{\geq 0}$ is a concave and increasing function which indicates the revenues from consuming resources, and $C_{xy}^{(t)}: \mathbb{R}_{\geq 0}\rightarrow\mathbb{R}_{\geq 0}$ is a convex and increasing function which captures the costs of acquiring resources from $y$. The superscript $(t)$ indicates that the function belongs to the party of targets, and it does not indicate exponentiation. We provide several examples of $U_{xy}^{(t)}$ and $C_{xy}^{(t)}$ and their interpretations for reference.  
\begin{itemize}
\item $U_{xy}^{(t)}:=\gamma_{xy}\pi_{xy}$ is a linear revenue function with $\gamma_{xy}$ denoting the unit revenue gained from consuming resources;
\item $U_{xy}^{(t)}:=\gamma_{xy}\min(\pi_{xy}, p_x)$ is a threshold revenue function. It captures the situation when $x$ has a threshold of consuming resources and $x$ cannot gain more revenues with more resources after the threshold.
\item $U_{xy}^{(t)}:=\gamma_{xy}\log(\pi_{xy}+1)$ is a concave $\log$ revenue function which captures situations when the revenues increase slower with more resources;
\item $C_{xy}^{(t)}:=c_{xy}^{(t)}\pi_{xy}$ indicates a linear cost function with $c_{xy}^{(t)}$ denoting the unit cost of acquiring resources from $y$.
\item $C_{xy}^{(t)}:=c_{xy}^{(t)}\pi_{xy}^2$ indicates a quadratic cost function which captures situations when the costs increase faster with more resources. 
\end{itemize}
\end{remark}
\begin{remark}
\label{rem:gxy}
The source $y$'s utility function $g_{xy}: \mathbb{R}_{\geq 0}\rightarrow\mathbb{R}$ can be expressed as follows:
\begin{equation}
\label{eq:remgxy}
g_{xy} = U_{xy}^{(s)}(\pi_{xy}) - C_{xy}^{(s)}(\pi_{xy}), 
\end{equation}
where $U_{xy}^{(s)}: \mathbb{R}_{\geq 0}\rightarrow\mathbb{R}_{\geq 0}$ is a concave and increasing function which indicates the incomes from providing resources to targets, and $C_{xy}^{(s)}: \mathbb{R}_{\geq 0}\rightarrow\mathbb{R}_{\geq 0}$ is a convex and increasing function which captures the losses caused by resource reduction. The superscript $(s)$ indicates that the function belongs to the party of sources, and it does not indicate exponentiation. We provide several examples of $U_{xy}^{(s)}$ and $C_{xy}^{(s)}$ and their interpretations for reference.  
\begin{itemize}
\item $U_{xy}^{(s)}:=\delta_{xy}\pi_{xy}$ indicates a linear income function with $\delta_{xy}$ denoting the unit income of providing resources to $x$;
\item $C_{xy}^{(s)}:=c_{xy}^{(s)}\pi_{xy}$ indicates a linear loss function and $c_{xy}^{(s)}$ can be interpreted as the unit loss of resource reduction.
\item $C_{xy}^{(s)}:=c_{xy}^{(s)}\pi_{xy}^2$ indicates a quadratic loss function which captures situations when the losses increase faster with more resources. 
\end{itemize}
\end{remark}
We have the following theorem regarding the existence of solution to problem  (\ref{eq:Main}). 
\begin{theorem}
\label{the:MainExist}
There exists a solution (not necessarily unique) for problem (\ref{eq:Main}) if it satisfies Condition \ref{con:MainNecessary} and has a feasible set. 
\end{theorem}
\begin{proof}
From Section 4.2 in \cite{boyd2004convex}, since both $f_{xy}$ and $g_{xy}$ are concave, there exists a solution for problem (\ref{eq:Main}) . However, since they may not be strictly concave, the solution may not be unique.
\end{proof}

Problem (\ref{eq:Main}) captures the central planner's objective of finding the optimal matching to transfer the resources between sources and targets. However, solving it requires a significant amount of computations when the number of participants is large. We aim to develop a fully distributed algorithm to find the optimal matching which does not require a central planner. We take a consensus-based approach and decentralize problem (\ref{eq:Main}) with alternating direction method of multipliers (ADMM).

\section{Distributed Discrete Optimal Transport Using ADMM}
\label{sec:DOT}
In the centralized discrete optimal transport problem (\ref{eq:Main}), the central planner designs ${\Pi_\mathcal{G}}$ for sources and targets, we could also interpret it as that sources and targets reach consensus with the central planner on ${\Pi_\mathcal{G}}$. Recall that $\pi_{xy}$ denotes the the amount of transferred resources from $y$ to $x$. Let $\pi_{xy}^{(t)}$ and $\pi_{xy}^{(s)}$ represent the amounts of resources that target $x$ requests from $y$ and source $y$ offers to $x$, respectively, the consensus between the participants and the central planner can be captured with $\pi_{xy}^{(t)} = \pi_{xy}$ and $\pi_{xy} = \pi_{xy}^{(s)}$. We further define ${\Pi_\mathcal{G}^{(t)}}:=\{\pi_{xy}^{(t)}|\{x, y\}\in\mathcal{G}\}$, ${\Pi_\mathcal{G}^{(s)}}:=\{\pi_{xy}^{(s)}|\{x, y\}\in\mathcal{G}\}$,  ${\Pi_{\overline{x}}^{(t)}}:=\{\pi_{xy}^{(t)}|y\in\mathcal{Y}_x, x = \overline{x}\}$, and ${\Pi_{\overline{y}}^{(s)}}:=\{\pi_{xy}^{(s)}|x\in\mathcal{X}_y, y = {\overline{y}}\}$. The superscripts $(t)$ and $(s)$ indicate that the variables or parameters belong to targets and sources, respectively, and they do not indicate exponentiation.

As a result, we can rewrite problem (\ref{eq:Main}) into the form of ADMM as the following problem:
\begin{equation}
\label{eq:MainReRe}
\begin{array}{c}
\min\limits_{\{ {\Pi_{\mathcal{G}}^{(t)}}\in\mathcal{U}_1,{\Pi_{\mathcal{G}}^{(s)}}\in\mathcal{U}_2,{\Pi_{\mathcal{G}}}\}} - \sum\limits_{x\in\mathcal{X}} \sum\limits_{y\in \mathcal{Y}_x} f_{xy}(\pi_{xy}^{(t)}) - \sum\limits_{y\in\mathcal{Y}  }  \sum\limits_{x\in \mathcal{X}_y}g_{xy}(\pi_{xy}^{(s)}) \\
\begin{array}{cc}
{\begin{array}{c}
\text{s.t.}
\end{array}}&{ \begin{array}{cc}
 {\pi_{xy}^{(t)} = \pi_{xy},}&{\forall \{x,y\}\in\mathcal{G},} \\  {\pi_{xy} = \pi_{xy}^{(s)},}&{\forall \{x,y\}\in\mathcal{G},}
\end{array}}
\end{array}
\end{array}
\end{equation}
where 
\[\mathcal{U}_1:=\left\lbrace {\Pi_{\mathcal{G}}^{(t)}}  | \pi_{xy}^{(t)} \geq 0,  p_{x, l} \leq \sum\limits_{y\in \mathcal{Y}_x}  \pi_{xy}^{(t)}  \leq p_{x, h},  \{x, y\} \in \mathcal{G} \right\rbrace,\] 
\[\mathcal{U}_2:=\left\lbrace {\Pi_{\mathcal{G}}^{(s)}} |\pi_{xy}^{(s)} \geq 0, q_{y, l} \leq \sum\limits_{x\in \mathcal{X}_y}  \pi_{xy}^{(s)} \leq q_{y, h},  \{x, y\} \in \mathcal{G} \right\rbrace.\] 
Problem (\ref{eq:MainReRe}) is a minimization problem with the objective function from problem (\ref{eq:Main}). Since both $f_{xy}$ and $g_{xy}$ are concave, the objective function of problem (\ref{eq:MainReRe}) is convex. The constraints in problem (\ref{eq:Main}) have been included in the feasible sets $\mathcal{U}_1$ and $\mathcal{U}_2$. The consensus constraints in problem (\ref{eq:MainReRe}) capture the consensus between the participants and the central planner. Furthermore, we can achieve $\pi_{xy}^{(t)} = \pi_{xy}^{(s)}$, which indicates the direct consensus between sources and targets. 

Let $\alpha_{xy1}$ and $\alpha_{xy2}$ denote the Lagrange multipliers corresponding to the constraints $\pi_{xy}^{(t)} = \pi_{xy}$ and $\pi_{xy} = \pi_{xy}^{(s)}$, respectively, and then the augmented Lagrangian related to problem (\ref{eq:MainReRe}) is
\begin{equation}
\label{eq:MainReReLagrange}
\begin{array}{l}
\mathcal{L}({\Pi_{\mathcal{G}}^{(t)}},{\Pi_{\mathcal{G}}^{(s)}},{\Pi_{\mathcal{G}}},\{\alpha_{xy1}\}_{\mathcal{G}}, \{\alpha_{xy2}\}_{\mathcal{G}}) \\ \ \ \ \
=- \sum\limits_{x\in\mathcal{X}}\sum\limits_{y\in \mathcal{Y}_x} f_{xy}( \pi_{xy}^{(t)} ) - \sum\limits_{y\in\mathcal{Y}}  \sum\limits_{x\in \mathcal{X}_y}  g_{xy}(\pi_{xy}^{(s)} ) \\ \ \ \ \ + \sum\limits_{x\in\mathcal{X}}\sum\limits_{y\in\mathcal{Y}_x} \alpha_{xy1}(\pi_{xy}^{(t)}-\pi_{xy}) 
+  \sum\limits_{y\in\mathcal{Y}}\sum\limits_{x\in\mathcal{X}_y} \alpha_{xy2}(\pi_{xy}-\pi_{xy}^{(s)}) \\ \ \ \ \ +\frac{\eta}{2}\sum\limits_{x\in\mathcal{X}}\sum\limits_{y\in\mathcal{Y}_x} (\pi_{xy}^{(t)}-\pi_{xy})^2  +\frac{\eta}{2}\sum\limits_{y\in\mathcal{Y}}\sum\limits_{x\in\mathcal{X}_y}  (\pi_{xy}-\pi_{xy}^{(s)})^2 .
\end{array}
\end{equation}
Note that $\mathcal{L}$ is strictly convex with the quadratic terms $\frac{\eta}{2}\sum_{x\in\mathcal{X}}\sum_{y\in\mathcal{Y}_x} (\pi_{xy}^{(t)}-\pi_{xy})^2 $ and $\frac{\eta}{2}\sum_{y\in\mathcal{Y}}\sum_{x\in\mathcal{X}_y} (\pi_{xy}-\pi_{xy}^{(s)})^2$, and thus guarantees the convergence to a unique optimum even when $f_{xy}$ and $g_{xy}$ are not strictly concave. The scalar $\eta$ controls the speed of the convergence. 

We can achieve the following proposition after applying the iterations of ADMM to problem (\ref{eq:MainReRe}).
\begin{proposition}
\label{pro:MainReRe}
{We can obtain the following iterations after applying the iterations of ADMM to Problem (\ref{eq:MainReRe})}:
\begin{equation}
\label{eq:MainReReSolUW}
\begin{array}{l}
{\Pi_{x}^{(t)}}(k+1)  \in \arg\min\limits_{{\Pi_{x}^{(t)}} \in \mathcal{U}_{x}} - \sum\limits_{y\in \mathcal{Y}_x} f_{xy}(\pi_{xy}^{(t)}) \\ \ \ \ \ \ \ \ \ \ \ \ \ + \sum\limits_{y\in\mathcal{Y}_x} \alpha_{xy1}(k)\pi_{xy}^{(t)}    + \frac{\eta}{2} \sum\limits_{y\in\mathcal{Y}_x}  (\pi_{xy}^{(t)} - \pi_{xy}(k))^2 ,
\end{array}
\end{equation}
\begin{equation}
\label{eq:MainReReSolVW}
\begin{array}{l}
{\Pi_{y}^{(s)}}(k+1) \in \arg\min\limits_{{\Pi_{y}^{(s)}} \in \mathcal{U}_{y}} - \sum\limits_{x\in \mathcal{X}_y} g_{xy}(\pi_{xy}^{(s)}) \\ \ \ \ \ \ \ \ \ \ \ \ \ - \sum\limits_{x\in\mathcal{X}_y}  \alpha_{xy2}(k)\pi_{xy}^{(s)}  + \frac{\eta}{2} \sum\limits_{x\in\mathcal{X}_y}  (\pi_{xy}(k)-\pi_{xy}^{(s)} )^2 ,
\end{array}
\end{equation}
\begin{equation}
\label{eq:MainReReSolBeta}
\begin{array}{l}
\pi_{xy}(k+1) = \arg\min\limits_{\pi_{xy}} -\alpha_{xy1}(k) \pi_{xy} + \alpha_{xy2}(k)\pi_{xy} \\  \ \ \ \ \    +\frac{\eta}{2}(\pi_{xy}^{(t)}(k+1)-\pi_{xy})^2 +\frac{\eta}{2}(\pi_{xy}-\pi_{xy}^{(s)}(k+1))^2 ,
\end{array}
\end{equation}
\begin{equation}
\label{eq:MainReReSolAlpha1}
\alpha_{xy1}(k+1) = \alpha_{xy1}(k) + \eta (\pi_{xy}^{(t)}(k+1) - \pi_{xy}(k+1)),
\end{equation}
\begin{equation}
\label{eq:MainReReSolAlpha2}
\alpha_{xy2}(k+1) = \alpha_{xy2}(k) + \eta (\pi_{xy}(k+1) -\pi_{xy}^{(s)}(k+1)),
\end{equation}
where $\mathcal{U}_x:=\{ {\Pi_x^{(t)}} | \pi_{xy}^{(t)} \geq 0, y\in\mathcal{Y}_x; p_{x, l} \leq \sum_{y\in \mathcal{Y}_x}  \pi_{xy}^{(t)}  \leq p_{x, h} \}$, $\mathcal{U}_y:=\{{\Pi_y^{(s)}} |\pi_{xy}^{(s)} \geq 0, x\in\mathcal{X}_y; q_{y, l} \leq \sum_{x\in \mathcal{X}_y}  \pi_{xy}^{(s)} \leq q_{y, h} \}$.
\end{proposition}
\begin{proof}
See Appendix B.
\end{proof}
Iterations (\ref{eq:MainReReSolUW})-(\ref{eq:MainReReSolAlpha2}) can be simplified further as shown in the following Proposition \ref{pro:MainReRe2}.
\begin{proposition}
\label{pro:MainReRe2}
{Iterations (\ref{eq:MainReReSolUW})-(\ref{eq:MainReReSolAlpha2}) can be simplified further into the following iterations:}
\begin{equation}
\label{eq:MainReReSol2UW}
\begin{array}{l}
{\Pi_{x}^{(t)}} (k+1)  \in \arg\min\limits_{{\Pi_{x}^{(t)}}  \in \mathcal{U}_{x}} - \sum\limits_{y\in \mathcal{Y}_x} f_{xy}(\pi_{xy}^{(t)}) \\ \ \ \ \ \ \ \ \ \ \ + \sum\limits_{y\in\mathcal{Y}_x} \alpha_{xy}(k)\pi_{xy}^{(t)}  + \frac{\eta}{2} \sum\limits_{y\in\mathcal{Y}_x}  (\pi_{xy}^{(t)} - \pi_{xy}(k))^2 ,
\end{array}
\end{equation}
\begin{equation}
\label{eq:MainReReSol2VW}
\begin{array}{l}
{\Pi_{y}^{(s)}} (k+1)  \in \arg\min\limits_{{\Pi_{y}^{(s)}} \in \mathcal{U}_{y}} - \sum\limits_{x\in \mathcal{X}_y} g_{xy} (\pi_{xy}^{(s)})\\ \ \ \ \ \ \ \ \ \ \  - \sum\limits_{x\in\mathcal{X}_y}  \alpha_{xy}(k)\pi_{xy}^{(s)} + \frac{\eta}{2} \sum\limits_{x\in\mathcal{X}_y}  (\pi_{xy}(k)-\pi_{xy}^{(s)} )^2 ,
\end{array}
\end{equation}
\begin{equation}
\label{eq:MainReReSol2Beta}
\begin{array}{l}
\pi_{xy}(k+1)   =\frac{1}{2} (\pi_{xy}^{(t)}(k+1) + \pi_{xy}^{(s)}(k+1)) ,
\end{array}
\end{equation}
\begin{equation}
\label{eq:MainReReSol2Alpha}
\alpha_{xy}(k+1) = \alpha_{xy}(k) + \frac{\eta}{2} (\pi_{xy}^{(t)}(k+1) - \pi_{xy}^{(s)}(k+1)),
\end{equation}
where $\mathcal{U}_x:=\{ {\Pi_{x}^{(t)}} | \pi_{xy}^{(t)} \geq 0, y\in\mathcal{Y}_x; p_{x, l} \leq \sum_{y\in \mathcal{Y}_x}  \pi_{xy}^{(t)}  \leq p_{x, h}  \}$, $\mathcal{U}_y:=\{{\Pi_{y}^{(s)}}  |\pi_{xy}^{(s)} \geq 0, x\in\mathcal{X}_y; q_{y, l} \leq \sum_{x\in \mathcal{X}_y}  \pi_{xy}^{(s)} \leq q_{y, h}  \}$.
\end{proposition}
\begin{proof}
Note that (\ref{eq:MainReReSolBeta}) can be solved directly as:
\begin{equation}
\label{eq:MainProofBeta}
\begin{array}{l}
\pi_{xy}(k+1) = \frac{1}{2\eta}(\alpha_{xy1}(k) - \alpha_{xy2}(k)) \\ \ \ \ \ \ \ \ \ \  \ \ \ \ \ \ \ \ \ \ \ \ \  + \frac{1}{2}(\pi_{xy}^{(t)}(k+1)+\pi_{xy}^{(s)}(k+1)).
\end{array}
\end{equation}
By plugging it into (\ref{eq:MainReReSolAlpha1}) and (\ref{eq:MainReReSolAlpha2}), we have that $
\alpha_{xy1}(k+1) = \frac{1}{2}(\alpha_{xy1}(k)+\alpha_{xy2}(k))+\frac{\eta}{2}(\pi_{xy}^{(t)}(k+1)-\pi_{xy}^{(s)}(k+1))$ and $
\alpha_{xy2}(k+1) = \frac{1}{2}(\alpha_{xy1}(k)+\alpha_{xy2}(k))+\frac{\eta}{2}(\pi_{xy}^{(t)}(k+1)-\pi_{xy}^{(s)}(k+1))$. Thus,  $\alpha_{xy1}(k) = \alpha_{xy2}(k)$ for $k> 0$, and we can achieve (\ref{eq:MainReReSol2Beta}) from (\ref{eq:MainProofBeta}). After writing $\alpha_{xy1}(k)$ and $\alpha_{xy2}(k)$ as $\alpha_{xy}(k)$, we have (\ref{eq:MainReReSol2UW}), (\ref{eq:MainReReSol2VW}), (\ref{eq:MainReReSol2Alpha}).
\end{proof}
Algorithm 1 summarizes (\ref{eq:MainReReSol2UW})-(\ref{eq:MainReReSol2Alpha}). Iterations  (\ref{eq:MainReReSol2UW}) and (\ref{eq:MainReReSol2VW}) are updates of target $x$ and source $y$, respectively. Iterations (\ref{eq:MainReReSol2Beta}) and (\ref{eq:MainReReSol2Alpha}) are updates of pair $xy$. Iterations (\ref{eq:MainReReSol2UW})-(\ref{eq:MainReReSol2Alpha}) are fully distributed iterations to solve problem (\ref{eq:MainReRe}).  A detailed analysis of Proposition \ref{pro:MainReRe2} and Algorithm 1 is provided in the following subsections. 

\begin{table}[]
	\label{alg:MainSol}
	\renewcommand{\arraystretch}{2.0}
	\centering
	\begin{tabular}{l}
		\hline
		\bfseries Algorithm 1 \\
		\hline
		1:\ \ \bf{for} $k=0,1,2,...$ do\\
		2:\ \ \ \ \ \ \ \ Each $x\in\mathcal{X}$ computes ${\Pi_{x}^{(t)}}(k+1)$ via (\ref{eq:MainReReSol2UW})\\
		3:\ \ \ \ \ \ \ \ Each $y\in\mathcal{Y}$ computes ${\Pi_{y}^{(s)}}(k+1)$ via (\ref{eq:MainReReSol2VW})\\
		4:\ \ \ \ \ \ \ \ Each pair $\{x,y\}\in\mathcal{G}$  computes $\pi_{xy}(k+1)$ via (\ref{eq:MainReReSol2Beta})\\
		5:\ \ \ \ \ \ \ \ Each pair $\{x,y\}\in\mathcal{G}$  computes $\alpha_{xy}(k+1)$ via (\ref{eq:MainReReSol2Alpha})\\
		6:\ \ \bf{end for}\\
		\hline
	\end{tabular}
\end{table}

\begin{figure}[]
\centering
\subfigure{\includegraphics[width=0.223\textwidth]{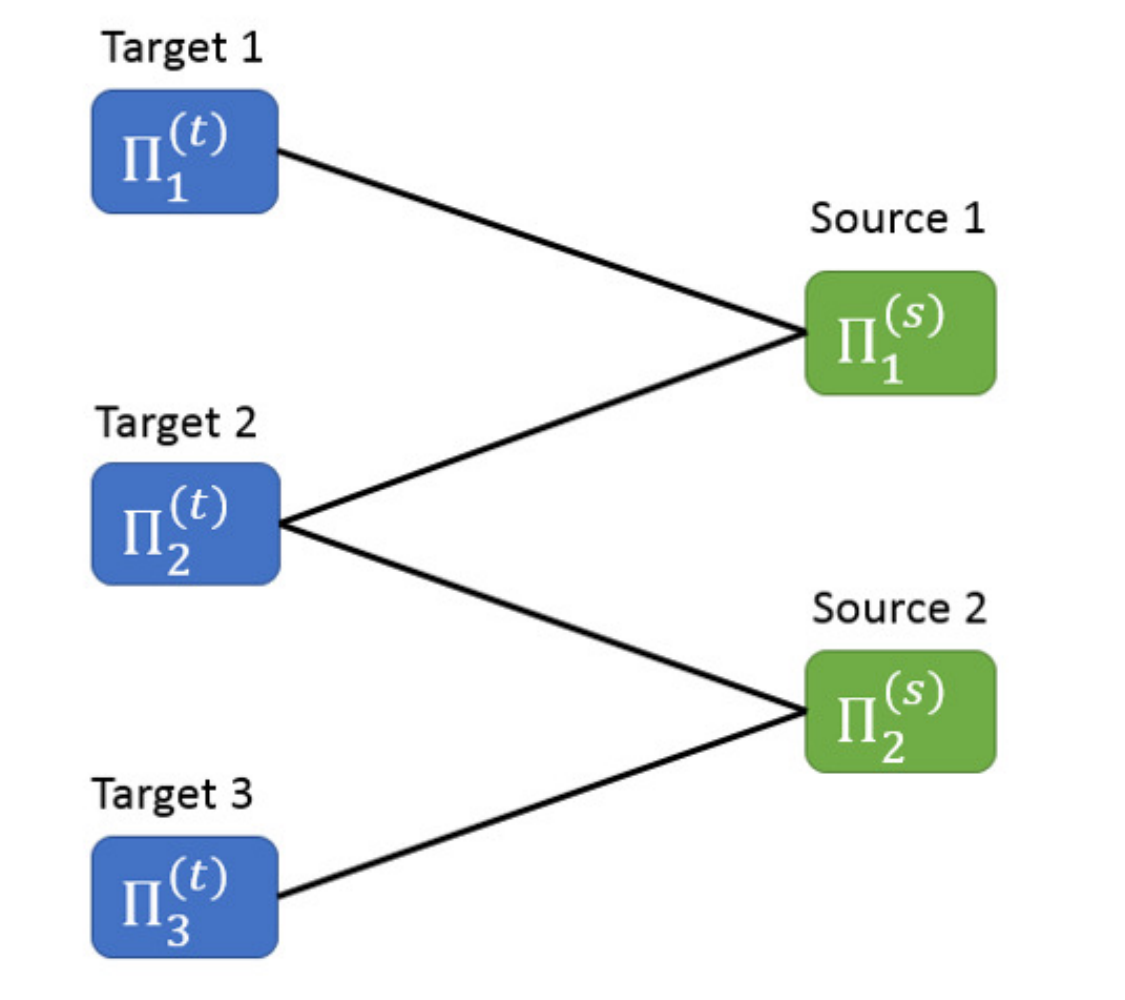}}
\subfigure{\includegraphics[width=0.223\textwidth]{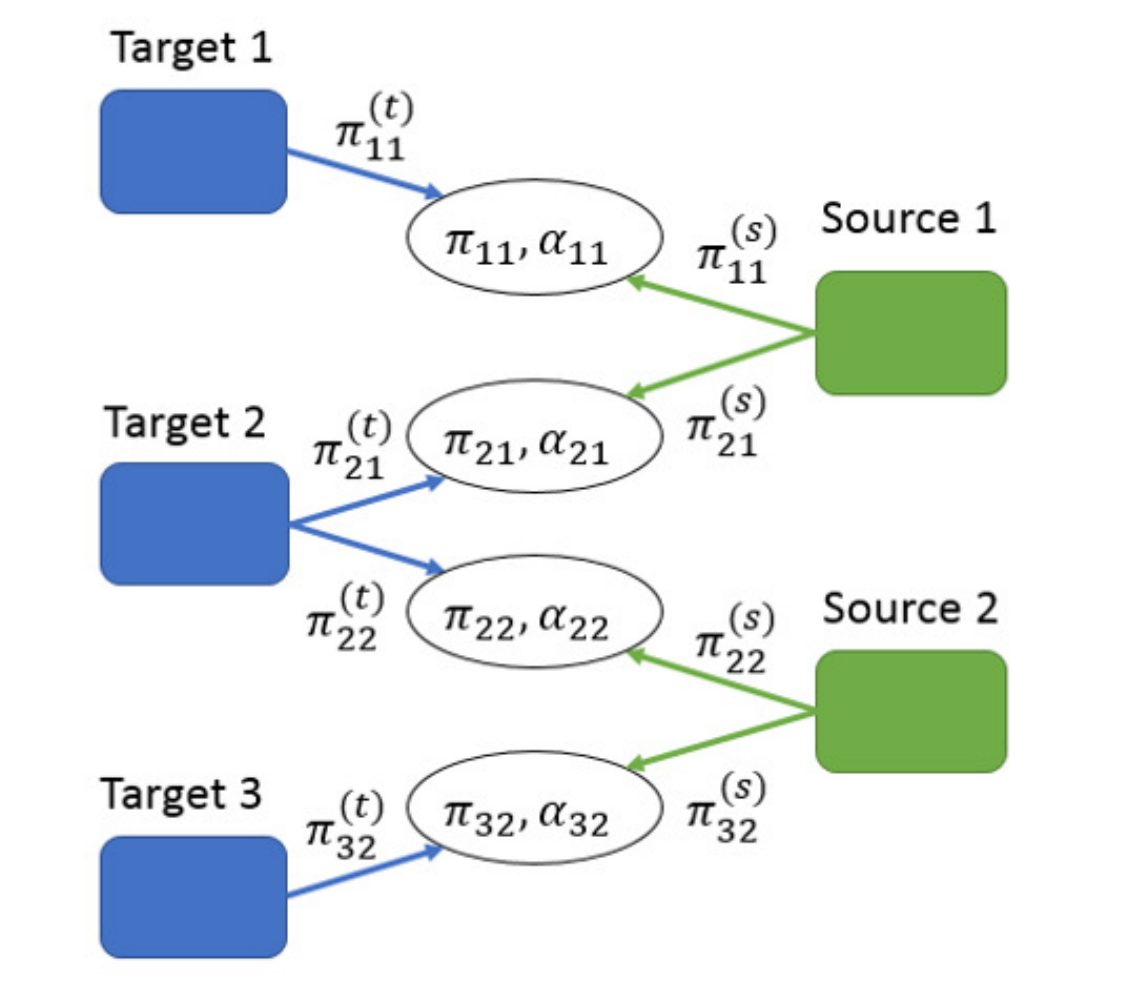}}
\caption{Illustration of Algorithm 1. At every iteration, each target $x$ or source $y$ achieves $\Pi_{x}^{(t)}$ or $\Pi_{y}^{(s)}$ with (\ref{eq:MainReReSol2UW}) or (\ref{eq:MainReReSol2VW}), respectively. Then, each pair of target $x$ and source $y$ computes $\pi_{xy}$ and $\alpha_{xy}$ with (\ref{eq:MainReReSol2Beta}) and (\ref{eq:MainReReSol2Alpha}), respectively. The iterations continue until convergence.}
\label{fig:Alg1}
\end{figure}

\subsection{Convergence and Complexity}
Since $f_{xy}$ and $g_{xy}$ are concave functions, we have that $-f_{xy}(\pi_{xy}^{(t)})-g_{xy}(\pi_{xy}^{(s)})$ is a convex function, and thus iterations (\ref{eq:MainReReSolUW})-(\ref{eq:MainReReSolAlpha2}) converge to the optimum of problem (\ref{eq:MainReRe}) from Lemma \ref{lem:ADMoMConvergence} in Appendix A. Since (\ref{eq:MainReReSol2UW})-(\ref{eq:MainReReSol2Alpha}) come from the simplifications of (\ref{eq:MainReReSolUW})-(\ref{eq:MainReReSolAlpha2}), we have that (\ref{eq:MainReReSol2UW})-(\ref{eq:MainReReSol2Alpha}) converge to the optimum of problem (\ref{eq:MainReRe}). Thus, the convergence of Algorithm 1 to the solution of problems (\ref{eq:Main}) and (\ref{eq:MainReRe}) is guaranteed.

The algorithm computes ${\Pi_{x}^{(t)}}$, ${\Pi_{y}^{(s)}}$, $\pi_{xy}$, and $\alpha_{xy}$ at each iteration. The complexities of computing ${\Pi_{x}^{(t)}}$ and ${\Pi_{y}^{(s)}}$ are determined by the utility functions $f_{xy}$ and $g_{xy}$ and the quadratic terms $\frac{\eta}{2} \sum_{y\in\mathcal{Y}_x}  (\pi_{xy}^{(t)} - \pi_{xy}(k))^2$ and $\frac{\eta}{2} \sum_{x\in\mathcal{X}_y}  (\pi_{xy}(k)-\pi_{xy}^{(s)} )^2 $, respectively. For example, if $g_{xy}$ is quadratic or linear, the complexity of computing ${\Pi_{y}^{(s)}}$ is equivalent to the complexity of quadratic programming which is $O(n^3)$, where $n$ is the size of ${\Pi_{y}^{(s)}}$. $\pi_{xy}$ and $\alpha_{xy}$ are obtained through matrix computations, which have the complexity of $O(1)$. The complexity of Algorithm 1 is further affected by the number of participants and connections between targets and sources.

\subsection{Efficiency and Privacy}
In Algorithm 1, each participant solves its own problem with (\ref{eq:MainReReSol2UW}) or (\ref{eq:MainReReSol2VW}), and each pair of target and source updates itself with (\ref{eq:MainReReSol2Beta}) and (\ref{eq:MainReReSol2Alpha}). Thus, it is a fully distributed algorithm which does not require central planners to collect the information of participants and compute the optimal matching. Algorithm 1 enables parallel computations and is more efficient for large-scale problems. Moreover, the only information that transmitted between participants are $\pi_{xy}^{(t)}$, $\pi_{xy}^{(s)}$, $\pi_{xy}$, and $\alpha_{xy}$, which reduces communication overhead and keeps privacy at the same time. Each participant's personal information $\{p_{x,l},p_{x,h}, f_{xy}, \mathcal{Y}_x\}$ or $\{q_{y,l},q_{y,h}, g_{xy}, \mathcal{X}_y\}$ is also kept private to itself. As a result, our algorithm guarantees certain levels of efficiency and privacy compared to centralized algorithms. 

\subsection{Resource Market: Negotiations and Bargaining}
In problem (\ref{eq:MainReRe}), the consensus constraints capture the consensus between participants and central planners. Furthermore, they can also be viewed as the direct consensus between sources and targets through central planners. One interpretation of the consensus constraints is that sources and targets must agree with each other to form a matching in resource markets. Algorithm 1 further reflects the negotiations and bargaining between sources and targets to reach agreements on the amount of transferred resources.  

$\pi_{xy}^{(t)}$ and $\pi_{xy}^{(s)}$ denote the amounts of resources that target $x$ requests from source $y$ and source $y$ offers to target $x$, respectively. They are proposed by $x$ and $y$ to meet their own objectives (\ref{eq:MainReReSol2UW}) and (\ref{eq:MainReReSol2VW}), respectively. $\pi_{xy}$ captures the consensus between $x$ and $y$ that they must reach an agreement on the amount of transferred resources, i.e., $\pi_{xy}^{(t)} = \pi_{xy} = \pi_{xy}^{(s)}$ from problem (\ref{eq:MainReRe}). We can see from (\ref{eq:MainReReSol2Beta}) that the agreement is approached by continuous negotiations following an averaging principle, i.e., $\pi_{xy}$ is taken as the average of $\pi_{xy}^{(t)}$ and $\pi_{xy}^{(s)}$. Moreover, we can see from (\ref{eq:MainReReSol2UW}) and (\ref{eq:MainReReSol2VW}) that $x$ and $y$ modify their proposals $\pi_{xy}^{(t)}$ and $\pi_{xy}^{(s)}$ to not only maximize their own utilities but also minimize the differences from their previous negotiations $\pi_{xy}$, respectively.  

$\alpha_{xy}$ is the Lagrange multipliers on the consensus constraints. We can see from (\ref{eq:MainReReSol2Alpha}) that $\alpha_{xy}$ captures cumulative differences between $x$ and $y$. A positive $\alpha_{xy}$ indicates that $x$ requests more than $y$ offers, thus, $x$ will decrease $\pi_{xy}^{(t)}$ with term $\alpha_{xy}(k)\pi_{xy}^{(t)}$ in (\ref{eq:MainReReSol2UW}), while $y$ will increase $\pi_{xy}^{(s)}$ with term $-\alpha_{xy}(k)\pi_{xy}^{(s)}$ in (\ref{eq:MainReReSol2VW}), and vice versa. $\alpha_{xy}$ identifies the directions for $x$ and $y$ to modify their proposals $\pi_{xy}^{(t)}$ and $\pi_{xy}^{(s)}$. It reveals the tendencies of sources and targets to reach agreements on the amounts of transferred resources between them. 

$\eta$ controls each participant's trade-off between its high utility and a small total difference from others. We can see from the last terms of the objective functions in (\ref{eq:MainReReSol2UW}) and (\ref{eq:MainReReSol2VW}) that a larger $\eta$ indicates that the participant cares more about reaching agreements with the participants on the other side, while a smaller $\eta$ indicates that the participant pays more attentions on maximizing its utility. Moreover, we can see from (\ref{eq:MainReReSol2Alpha}) that a large $\eta$ amplifies the degree of disagreement $\alpha_{xy}$, which will increase the costs of disagreement for the participant through the second terms of the objective functions in (\ref{eq:MainReReSol2UW}) and (\ref{eq:MainReReSol2VW}). 

As a result, our algorithm reflects the negotiation and bargaining processes between targets and sources in market environments. Each participant aims to maximize its own utility and minimize the differences from others. The matching is reached by the bargaining between sources and targets on the amounts of transferred resources following an averaging principle. The averaging principle contains three aspects: 1) each pair of source and target always settles at their average amount of transferred resources after each negotiation; 2) the source and the target then modify their proposals by reducing their differences from the average amount of resources from their last round of negotiation; 3) the negotiations repeat until agreements are reached for all sources and targets. The averaging principle provides a way to regulate resource markets and improve resource efficiency as the convergence of our algorithm to the optimal matching is guaranteed. 

\subsection{Online Distributed Resource Matching}
In many applications, targets and sources may change over time. For example, a source may increase its upper bound $q_{y, h}$ when it has more resources to distribute; a target may achieve more utilities with the same amount of resources after improving the efficiency of resource consumption; new participants may join the matching and old participants may quit the matching. Our algorithm can be extended to its online form to address the real-time requirements for these applications.  
\begin{remark}
\label{rem:MainReReOnline}
Let $\mathcal{X}^{(k)}$, $\mathcal{Y}^{(k)}$, $f_{xy}^{(k)}$, $g_{xy}^{(k)}$, $\mathcal{U}_x^{(k)}$, and $\mathcal{U}_y^{(k)}$ denote the discrete optimal transport problem at time $k$, the iterations of finding the optimal matching are:
\begin{equation}
\label{eq:MainReReSol2UWOnline}
\begin{array}{l}
{\Pi_{x}^{(t)}}(k+1)  \in \arg\min\limits_{{\Pi_{x}^{(t)}} \in \mathcal{U}_{x}^{(k+1)}   } - \sum\limits_{y\in \mathcal{Y}_x ^{(k+1)} } f_{xy}^{(k+1)}(\pi_{xy}^{(t)}) \\ \ \ \ \ \ \  + \sum\limits_{y\in\mathcal{Y}_x ^{(k+1)} } \alpha_{xy}(k)\pi_{xy}^{(t)}  + \frac{\eta}{2} \sum\limits_{y\in\mathcal{Y}_x ^{(k+1)}}  (\pi_{xy}^{(t)} - \pi_{xy}(k))^2 ,
\end{array}
\end{equation}
\begin{equation}
\label{eq:MainReReSol2VWOnline}
\begin{array}{l}
{\Pi_{y}^{(s)}}(k+1)  \in \arg\min\limits_{{\Pi_{y}^{(s)}} \in \mathcal{U}_{y}^{(k+1)} } - \sum\limits_{x\in \mathcal{X}_y^{(k+1)} } g_{xy}^{(k+1)} (\pi_{xy}^{(s)})\\ \ \ \ \ \ \ \ - \sum\limits_{x\in\mathcal{X}_y^{(k+1)} }  \alpha_{xy}(k)\pi_{xy}^{(s)} + \frac{\eta}{2} \sum\limits_{x\in\mathcal{X}_y^{(k+1)}  }  (\pi_{xy}(k)-\pi_{xy}^{(s)} )^2 ,
\end{array}
\end{equation}
\begin{equation}
\label{eq:MainReReSol2BetaOnline}
\begin{array}{l}
\pi_{xy}(k+1)   =\frac{1}{2} (\pi_{xy}^{(t)}(k+1) + \pi_{xy}^{(s)}(k+1)) ,
\end{array}
\end{equation}
\begin{equation}
\label{eq:MainReReSol2AlphaOnline} 
\alpha_{xy}(k+1) = \alpha_{xy}(k) + \frac{\eta}{2} (\pi_{xy}^{(t)}(k+1) - \pi_{xy}^{(s)}(k+1)),
\end{equation}
where $\mathcal{U}_x^{(k+1)}:=\{ {\Pi_{x}^{(t)}} | \pi_{xy}^{(t)} \geq 0, y\in\mathcal{Y}_x ^{(k+1)} ; p_{x, l}^{(k+1)} \leq \sum_{y\in \mathcal{Y}_x^{(k+1)} }  \pi_{xy}^{(t)}  \leq p_{x, h}^{(k+1)} \}$, $\mathcal{U}_y^{(k+1)}:=\{{\Pi_{y}^{(s)}} |\pi_{xy}^{(s)} \geq 0, x\in\mathcal{X}_y^{(k+1)}; q_{y,l}^{(k+1)} \leq \sum_{x\in \mathcal{X}_y^{(k+1)}}  \pi_{xy}^{(s)} \leq q_{y,h}^{(k+1)} \}$.
\end{remark}
We can see from Remark \ref{rem:MainReReOnline} that each participant and each pair of participants update based on the current discrete optimal transport problem. We can continue finding the optimal matching without rerunning  the online algorithm when the discrete optimal transport problem changes over time. Note that the iterations in Remark \ref{rem:MainReReOnline} will always converge to the solution of the updated discrete optimal transport problem. However, if the problem changes quickly over time, the iterations may not find any optimal matching. One way to address this issue is to run a sufficient number of iterations to achieve the optimal matching of the current problem, and then update the problem and find the next optimal matching.

\section{Duality and Distributed Dual Discrete Optimal Transport With Linear Utilities}
\label{sec:DD}
In this section, we consider a special case of discrete optimal transport problems with linear objective functions and equality constraints, i.e., $f_{xy}(\pi_{xy}) = \gamma_{xy}\pi_{xy}$, $g_{xy}(\pi_{xy}) = \delta_{xy}\pi_{xy}$, $p_{x,l} = p_{x, h}=p_x$, and $q_{y, l} = q_{y, h}=q_y$. {Parameters $\gamma_{xy}\in\mathbb{R}_{\geq 0}$ and $\delta_{xy}\in\mathbb{R}_{\geq 0}$ are the unit utilities for $x$ and $y$, respectively.} Inequalities $p_{x, l} \leq \sum_{y\in \mathcal{Y}_x}  \pi_{xy}  \leq p_{x, h}$ and $q_{y, l} \leq \sum_{x\in \mathcal{X}_y}  \pi_{xy} \leq q_{y, h}$ become equalities $ \sum_{y\in \mathcal{Y}_x}  \pi_{xy}  = p_{x}$ and $\sum_{x\in \mathcal{X}_y}  \pi_{xy} = q_{y}$, respectively. Moreover, we assume that all sources match with all targets, i.e., $\mathcal{Y}_x = \mathcal{Y}$ and $\mathcal{X}_y = \mathcal{X}$. 

The discrete optimal transport in this case aims to find the optimal matching by solving the following problem:
\begin{equation}
\label{eq:OptimalTransport}
\begin{array}{c}
\max\limits_{ \{\pi_{xy} | \pi_{xy} \geq 0,  \{x,y\}\in\mathcal{G}\}  } \sum\limits_{x\in\mathcal{X}} \sum\limits_{y\in\mathcal{Y}} \pi_{xy} \gamma_{xy} +  \sum\limits_{y\in\mathcal{Y}}\sum\limits_{x\in\mathcal{X}}\pi_{xy}\delta_{xy} \\
\begin{array}{ccc}
{\text{s.t.}}&{\sum\limits_{y\in\mathcal{Y}} \pi_{xy} = p_x, \sum\limits_{x\in\mathcal{X}} \pi_{xy} =q_y,}&{\forall x\in\mathcal{X},y\in\mathcal{Y}.}
\end{array}
\end{array}
\end{equation}
Problem (\ref{eq:OptimalTransport}) is a well-studied problem in optimal transport \cite{villani2008optimal, galichon2015optimal}. We have the following condition regarding the existence of its solution which captures a situation when supply meets demand. 
\begin{condition}
\label{con:exisitanceOT}
There exists a solution of problem (\ref{eq:OptimalTransport}), which is not necessarily unique, if and only if the following equality holds:
\begin{equation}
\sum\limits_{x\in\mathcal{X}} p_x = \sum\limits_{y\in\mathcal{Y}} q_y. 
\end{equation}
\end{condition}
We can easily verify that Condition \ref{con:exisitanceOT} satisfies Condition \ref{con:MainNecessary}. The optimal value, i.e., maximum total surplus, is unique, which is also referred as the optimal transport distance:
\begin{equation}
\label{eq:OptimalTransportDistance}
\begin{array}{l}
D_1 :=
\\ 
\sup\limits_{\{\pi_{xy}\geq 0\} } \left\{\sum\limits_{x\in\mathcal{X}} \sum\limits_{y\in\mathcal{Y}} \pi_{xy} (\gamma_{xy} + \delta_{xy} ) \Big|  \sum\limits_{y\in\mathcal{Y}} \pi_{xy} = p_x, \sum\limits_{x\in\mathcal{X}} \pi_{xy} =q_y.   \right\}.
\end{array}
\end{equation}
After introducing $\pi_{xy}^{(t)}$ and $\pi_{xy}^{(s)}$ with a similar approach in the previous section, we can rewrite problem (\ref{eq:OptimalTransport}) as follows:
\begin{equation}
\label{eq:OTReRe}
\begin{array}{c}
\min\limits_{\{{\Pi_\mathcal{G}}^{(t)}\in\mathcal{U}_1,{\Pi_\mathcal{G}^{(s)}}\in\mathcal{U}_2,{\Pi_\mathcal{G}}\}} - \sum\limits_{x\in\mathcal{X}} \sum\limits_{y\in \mathcal{Y}} \pi_{xy}^{(t)}\gamma_{xy} - \sum\limits_{y\in\mathcal{Y}  }  \sum\limits_{x\in \mathcal{X}}\pi_{xy}^{(s)}\delta_{xy} \\
\begin{array}{cc}
{\begin{array}{c}
	\text{s.t.}
	\end{array}}&{ \begin{array}{cc}
	{\pi_{xy}^{(t)} = \pi_{xy},}&{\forall \{x,y\}\in\mathcal{G},} \\  {\pi_{xy} = \pi_{xy}^{(s)},}&{\forall \{x,y\}\in\mathcal{G},}
	\end{array}}
\end{array}
\end{array}
\end{equation}
where $\mathcal{U}_1:=\left\lbrace {\Pi_\mathcal{G}^{(t)}} | \pi_{xy}^{(t)} \geq 0,   \sum_{y\in \mathcal{Y}}  \pi_{xy}^{(t)}  =p_x, \{x,y\} \in \mathcal{G} \right\rbrace$ and $ 
\mathcal{U}_2:=\left\lbrace {\Pi_\mathcal{G}^{(s)}}  |\pi_{xy}^{(s)} \geq 0, \sum_{x\in \mathcal{X}}  \pi_{xy}^{(s)} = q_y, \{x,y\} \in \mathcal{G} \right\rbrace$.
By Proposition \ref{pro:MainReRe2}, we have:
\begin{proposition}
	\label{pro:OTRePrimal2}
	{We can obtain the following iterations after applying the iterations of ADMM to Problem (\ref{eq:OTReRe})}:
	\begin{equation}
	\label{eq:OTRePrimalCon2Sol2UW}
	\begin{array}{l}
	{\Pi_x^{(t)}}(k+1)  \in \arg\min\limits_{{\Pi_x^{(t)}} \in \mathcal{U}_{x}} - \sum\limits_{y\in \mathcal{Y}}  \pi_{xy}^{(t)} \gamma_{xy} \\ \ \ \ \ \ \ \ \ \ \  + \sum\limits_{y\in\mathcal{Y}} \alpha_{xy}(k)\pi_{xy}^{(t)}  + \frac{\eta}{2} \sum\limits_{y\in\mathcal{Y}}  (\pi_{xy}^{(t)} - \pi_{xy}(k))^2 ,
	\end{array}
	\end{equation}
	\begin{equation}
	\label{eq:OTRePrimalCon2Sol2VW}
	\begin{array}{l}
	{\Pi_y^{(s)}}(k+1)   \in \arg\min\limits_{{\Pi_y^{(s)}}  \in \mathcal{U}_{y}} - \sum\limits_{x\in \mathcal{X}}\pi_{xy}^{(s)} \delta_{xy} \\ \ \ \ \ \ \ \ \ \ \  - \sum\limits_{x\in\mathcal{X}}  \alpha_{xy}(k)\pi_{xy}^{(s)}+ \frac{\eta}{2} \sum\limits_{x\in\mathcal{X}}  (\pi_{xy}(k)-\pi_{xy}^{(s)} )^2 ,
	\end{array}
	\end{equation}
	\begin{equation}
	\label{eq:OTRePrimalCon2Sol2Beta}
	\begin{array}{l}
	\pi_{xy}(k+1)   =\frac{1}{2} (\pi_{xy}^{(t)}(k+1) + \pi_{xy}^{(s)}(k+1)) ,
	\end{array}
	\end{equation}
	\begin{equation}
	\label{eq:OTRePrimalCon2Sol2Alpha}
	\alpha_{xy}(k+1) = \alpha_{xy}(k) + \frac{\eta}{2} (\pi_{xy}^{(t)}(k+1) - \pi_{xy}^{(s)}(k+1)),
	\end{equation}
	where $\mathcal{U}_x:=\{ {\Pi_x^{(t)}} | \pi_{xy}^{(t)} \geq 0, y\in\mathcal{Y}; \sum_{y\in \mathcal{Y}}  \pi_{xy}^{(t)}  = p_x \}$, $\mathcal{U}_y:=\{{\Pi_y^{(s)}} |\pi_{xy}^{(s)} \geq 0,  x\in\mathcal{X}; \sum_{x\in \mathcal{X}}  \pi_{xy}^{(s)} = q_y \}$.
\end{proposition}
\begin{proof}
We can find the iterations by letting $f_{xy}(\pi_{xy}^{(t)}) \rightarrow \pi_{xy}^{(t)}\gamma_{xy}$, $g_{xy}(\pi_{xy}^{(s)}) \rightarrow \pi_{xy}^{(s)}\delta_{xy}$, $p_{x, l} \leq \sum_{y\in \mathcal{Y}_x}  \pi_{xy}^{(t)}  \leq p_{x, h} \rightarrow \sum_{y\in \mathcal{Y}}  \pi_{xy}  ^{(t)}= p_{x} $, and $q_{y, l} \leq \sum_{x\in \mathcal{X}_y}  \pi_{xy}^{(s)} \leq q_{y, h} \rightarrow \sum_{x\in \mathcal{X}}  \pi_{xy}^{(s)} = q_{y}$.
\end{proof}

One fundamental result of problems (\ref{eq:OptimalTransport}) and (\ref{eq:OTReRe}) is the Monge-Kantorovich duality\cite{galichon2015optimal}, which finds the dual problem of discrete optimal transport and shows no duality gap between the primal problem and the dual problem. The dual problem enriches the economic interpretations and reveals the pricing scheme of resource matching. In the following subsections, we discuss in detail the dual problem of problem (\ref{eq:OTReRe}), and we present a distributed dual algorithm which is found by solving the dual problem with ADMM. We prove that the iterations of the dual algorithm is equivalent to the iterations of the primal algorithm in Proposition \ref{pro:OTRePrimal2}. Moreover, we show that $\alpha_{xy}$ in Proposition \ref{pro:OTRePrimal2} captures the price that target $x$ pays to source $y$ as the compensation of transferring resources.  
\subsection{Dual Problem of Discrete Optimal Transport}
We derive the dual problem of problem (\ref{eq:OTReRe}) which can be used to develop the distributed dual algorithm of discrete optimal transport and study the pricing scheme of resource matching. 
\begin{theorem}
\label{the:OTReDual}
Let {$u_x\in\mathbb{R}$, $v_y\in\mathbb{R}$, $w_{xy}^{(t)}\in\mathbb{R}$, and $w_{xy}^{(s)}\in\mathbb{R}$} denote the Lagrange multipliers with respect to $\sum_{y\in\mathcal{Y}} \pi_{xy}^{(t)}= p_x$, $\sum_{x \in\mathcal{X}} \pi_{xy}^{(s)} = q_y$, $\pi_{xy}^{(t)} = \pi_{xy}$, and $\pi_{xy} = \pi_{xy}^{(s)}$, respectively, the dual problem of problem (\ref{eq:OTReRe}) is
\begin{equation}
\label{eq:OTReDual}
\begin{array}{c}
\min\limits_{\{{\mathbf{u}_\mathcal{X}},{\mathbf{v}_{\mathcal{Y}}},{\mathbf{w}_{\mathcal{G}}^{(t)}},{\mathbf{w}_{\mathcal{G}}^{(s)}}\}}\sum\limits_{x\in\mathcal{X}}u_x p_x + \sum\limits_{y\in\mathcal{Y}} v_y q_y\\
\begin{array}{cc}
{\begin{array}{c}
	\text{s.t.}\\ \ 
	\end{array}}&{ \begin{array}{cc}
	{ \gamma_{xy} - u_x - w_{xy}^{(t)} \leq  0,}&{\forall \{x,y\}\in\mathcal{G};}\\
	{\delta_{xy} - v_y  + w_{xy}^{(s)} \leq 0,}&{\forall \{x,y\}\in\mathcal{G},}\\
	{w_{xy}^{(t)} = w_{xy}^{(s)},}&{\forall \{x,y\}\in\mathcal{G}.}
	\end{array}}
\end{array}
\end{array}
\end{equation}
\end{theorem}
\begin{proof}
See Appendix C.
\end{proof}
We summarize several interesting results of problems (\ref{eq:OTReRe}) and (\ref{eq:OTReDual}) with the following proposition.
\begin{proposition}
	\label{pro:OTRe}
	The following facts hold for the primal problem (\ref{eq:OTReRe}) and the dual problem (\ref{eq:OTReDual}):
	\begin{itemize}
		\item[(i)] Problem (\ref{eq:OTReRe}) is the dual problem of problem (\ref{eq:OTReDual}).
		\item[(ii)] (Monge-Kantorovich duality) The optimal value of problem  (\ref{eq:OTReRe}) is equal to the optimal value of problem (\ref{eq:OTReDual}). The optimal transport distance (\ref{eq:OptimalTransportDistance}) can also be written as
		\begin{equation}
		\label{eq:OTReDualDistance}
		\begin{array}{l}
		D_2:=
		\\ \inf\limits_{\{{\mathbf{u},\mathbf{v},\mathbf{w}^{(t)},\mathbf{w}^{(s)}}\}} \left\{  \sum\limits_{x\in\mathcal{X}}u_x p_x +\sum\limits_{y\in\mathcal{Y}} v_y q_y \left| \begin{array}{c}
		\gamma_{xy} - u_x - w_{xy}^{(t)} \leq 0,\\ 
		\delta_{xy} - v_y  + w_{xy}^{(s)} \leq 0,\\
		w_{xy}^{(t)} = w_{xy}^{(s)}.
		\end{array}  \right .\right\}.
		\end{array}
		\end{equation}
		\item[(iii)] The solution $\{{\Pi_\mathcal{G}^*}, {\Pi_\mathcal{G}^{(t)*}},{\Pi_\mathcal{G}^{(s)*}}\}$ of problem (\ref{eq:OTReRe}) may not be unique but $\pi_{xy}^*=\pi_{xy}^{(t)*} =\pi_{xy}^{(s)*}$ always holds; the solution $\{u_x^*,v_y^*, w_{xy}^{(t)*},w_{xy}^{(s)*}\}$ of problem (\ref{eq:OTReDual}) is not unique but $w_{xy}^{(t)*} =w_{xy}^{(s)*}$ always holds. Let $C\in\mathbb{R}$ denote a constant, $\{u_x^*+C,v_y^*-C, w_{xy}^{(t)*}-C,w_{xy}^{(s)*}-C\}$ is also a solution of problem (\ref{eq:OTReDual}). 	
		\item[(iv)] if $\pi_{xy}^{(t)*}> 0$ or $\pi_{xy}^{(s)*}> 0$, we have $u_x^*= \gamma_{xy}-w_{xy}^{(t)*}$, $v_y^* = \delta_{xy}+w_{xy}^{(s)*}$, and $u_x^* + v_y^* = \gamma_{xy} + \delta_{xy}$. 
		\item[(v)] $u_x^*,v_y^*$, $w_{xy}^{(t)*}$, and $w_{xy}^{(s)*}$ satisfy: 
		\begin{equation}
		\label{eq:OTReDualSol}
		\begin{array}{cc}
		{u_x^* = \max\limits_{y\in\mathcal{Y}} \{\gamma_{xy} - w_{xy}^{(t)*} \}, }&{v_y^* = \max\limits_{x\in\mathcal{X}} \{\delta_{xy} + w_{xy}^{(s)*} \}.} 
		\end{array}
		\end{equation}
	\end{itemize}
\end{proposition}
\begin{proof}
We can verify (i) by deriving the dual problem of problem (\ref{eq:OTReDual}). (ii) can be achieved by the minimax theorem \cite{sion1958general}. To prove (iii), we notice that $\gamma_{xy}-(u_x^*+C)-(w_{xy}^{(t)*}-C) = \gamma_{xy}-u_x^*-w_{xy}^{(t)*}$, $\delta_{xy}-(v_y^*-C)+(w_{xy}^{(s)*}-C) = \delta_{xy}-v_y^*+w_{xy}^{(s)*}$, $w_{xy}^{(t)*}-C=w_{xy}^{(s)*}-C$, and $\sum_{x\in\mathcal{X}}(u_x^*+C) p_x +\sum_{y\in\mathcal{Y}} (v_y^*-C) q_y = \sum_{x\in\mathcal{X}}u_x^* p_x +\sum_{y\in\mathcal{Y}} v_y^* q_y + \sum_{x\in\mathcal{X}}C p_x -\sum_{y\in\mathcal{Y}} C q_y =  \sum_{x\in\mathcal{X}}u_x^* p_x +\sum_{y\in\mathcal{Y}} v_y^* q_y $ as $\sum_{x\in\mathcal{X}}p_x = \sum_{y\in\mathcal{Y}} q_y$ in Condition \ref{con:exisitanceOT}. Thus, (iii) holds. We can obtain (iv) by complementary slackness \cite{dantzig1998linear}. To prove (v), we first notice that the constraints in problem (\ref{eq:OTReDual}) imply that $u_x^* \geq \max\limits_{y} \{\gamma_{xy} - w_{xy}^{(t)*}\}$ and $v_y^* \geq \max\limits_{x} \{\delta_{xy} + w_{xy}^{(s)*}\}$. However, if the inequalities hold strict, $u_x^*$ and $v_y^*$ will not be optimal as one can achieve a lower value of  $\sum_{x\in\mathcal{X}}u_x^* p_x +\sum_{y\in\mathcal{Y}} v_y^* q_y$ by selecting $u_x^* = \max\limits_{y} \{\gamma_{xy} - w_{xy}^{(t)*}\}$ and $v_y^*=\max\limits_{x} \{\delta_{xy} + w_{xy}^{(s)*}\}$, and thus the inequalities are actually equalities, which proves (v).
\end{proof}
The dual problem and Monge-Kantorovic duality has largely enriched the interpretation of optimal transport \cite{galichon2015optimal}. The dual variables $u_x$ and $v_y$ can be interpreted as the individual surpluses that $x$ and $y$ achieve in their matching with others, respectively; the dual consensus variables $w_{xy}^{(t)}$ and $w_{xy}^{(s)}$ can be interpreted as the proposed prices that $x$ pays and $y$ charges for transferring resources, respectively. Proposition \ref{pro:OTRe}(iv) indicates that a matching between $x$ and $y$ exists only if their individual surpluses $u_x+v_y$ are equal to their pair surpluses $\gamma_{xy} + \delta_{xy}$, which is also reflected in (\ref{eq:OTReDual}) on the minimization over the overall individual surpluses under constraints $u_x\geq \gamma_{xy}-w_{xy}^{(t)}$, $v_y\geq \delta_{xy}+w_{xy}^{(s)}$, and $w_{xy}^{(t)}=w_{xy}^{(s)}$, i.e., $u_x+v_y\geq \gamma_{xy}+\delta_{xy}$. The prices $w_{xy}^{(t)}$ and $w_{xy}^{(s)}$ only serve as the compensation for transferring resources and they do not change the total surplus $u_x^*+v_y^*=\gamma_{xy} + \delta_{xy}$ of $x$ and $y$. We can see from Proposition \ref{pro:OTRe}(v) that the objective of each participant is to maximize its individual surplus by taking into account both its pair surpluses and the prices it pays or charges.

The objective functions of problems (\ref{eq:OTReRe}) and (\ref{eq:OTReDual}) capture the total surplus by summing over pair surpluses and individual surpluses, respectively. On the one hand, we can see from problem (\ref{eq:OTReRe}) that the primal problem of discrete optimal transport aims to find the optimal matching rules that maximize the total surplus of all pairs between $\mathcal{X}$ and $\mathcal{Y}$; on the other hand, we can see from Proposition \ref{pro:OTRe}(v) that the dual problem of discrete optimal transport aims to find an equilibrium where all participants reach their highest individual surpluses. The Monge-Kantorovich duality in Proposition \ref{pro:OTRe}(ii) indicates that both interpretations return the same optimal values and reach the maximum total surplus of all sources and targets. 

\subsection{Distributed Dual Algorithm}
We have obtained a fully distributed primal algorithm to solve problem (\ref{eq:OTReRe}) with ADMM as shown in Proposition \ref{pro:OTRePrimal2}, and we can also obtain a fully distributed dual algorithm to solve the dual problem (\ref{eq:OTReDual}) using similar methods. We first introduce consensus variables $w_{xy}$ to decompose $w_{xy}^{(t)} = w_{xy}^{(s)}$ as $w_{xy}^{(t)} = w_{xy}$ and $w_{xy} = w_{xy}^{(s)}$. Thus, problem (\ref{eq:OTReDual}) can be rewritten as
\begin{equation}
\label{eq:OTReDualCon}
\begin{array}{c}
\min\limits_{\{ \{ {\mathbf{u}_\mathcal{X}},{\mathbf{w}_\mathcal{G}^{(t)}}\}\in\mathcal{U}_1',\{{\mathbf{v}_\mathcal{Y}},{\mathbf{w}_{\mathcal{G}}^{(s)}}\}\in\mathcal{U}_2',{\mathbf{w}_\mathcal{G}}\}} \sum\limits_{x\in\mathcal{X}} u_x p_x + \sum\limits_{y\in\mathcal{Y}} v_y q_y\\
\begin{array}{cc}
{\begin{array}{c}
\text{s.t.}
\end{array}}&{ \begin{array}{cc}
 {w_{xy}^{(t)} = w_{xy}, w_{xy} = w_{xy}^{(s)},}&{\forall \{x,y\}\in\mathcal{G},}
\end{array}}
\end{array}
\end{array}
\end{equation}
where $\mathcal{U}_1':=\{ {\mathbf{u}_\mathcal{X}},{\mathbf{w}_\mathcal{G}^{(t)}} | \gamma_{xy} - u_x - w_{xy}^{(t)} \leq 0,  \{x,y\}\in\mathcal{G}  \}$, and $\mathcal{U}_2':=\{{\mathbf{v}_\mathcal{Y}},{\mathbf{w}_\mathcal{G}^{(s)}} |  \delta_{xy} - v_y + w_{xy}^{(s)} \leq 0,  \{x,y\}\in\mathcal{G}  \}$.
\begin{proposition}
\label{pro:OTRe2}
{We can obtain the following iterations after applying the iterations of ADMM to Problem (\ref{eq:OTReDualCon})}:
\begin{equation}
\label{eq:OTReDualCon2Sol2UW}
\begin{array}{l}
\left\{u_x(k+1),{\mathbf{w}_x^{(t)}}(k+1)  \right\}   \in \arg\min\limits_{\{u_x,{\mathbf{w}_{x}^{(t)}}\} \in \mathcal{U}_{x}}  u_xp_x \\ \ \ \ \ \ \ \ \ \ \  + \sum\limits_{y\in\mathcal{Y}} \beta_{xy}(k)w_{xy}^{(t)}   + \frac{\widehat{\eta}}{2} \sum\limits_{y\in\mathcal{Y}}\left(w_{xy}^{(t)} - w_{xy}(k)  \right)^2 ,
\end{array}
\end{equation}
\begin{equation}
\label{eq:OTReDualCon2Sol2VW}
\begin{array}{l}
\left\{v_y(k+1),{\mathbf{w}_y^{(s)}}(k+1)\right\}  \in \arg\min\limits_{\{v_y,{\mathbf{w}_y^{(s)}} \} \in \mathcal{U}_{y}}  v_yq_y \\ \ \ \ \ \ \ \ \ \ \ - \sum\limits_{x\in\mathcal{X}} \beta_{xy}(k)w_{xy}^{(s)}  + \frac{\widehat{\eta}}{2}  \sum\limits_{x\in\mathcal{X}} \left(w_{xy}(k) -w_{xy}^{(s)} \right)^2 ,
\end{array}
\end{equation}
\begin{equation}
\label{eq:OTReDualCon2Sol2Beta}
\begin{array}{l}
w_{xy}(k+1) = \frac{1}{2} (w_{xy}^{(t)}(k+1) + w_{xy}^{(s)}(k+1)),
\end{array}
\end{equation}
\begin{equation}
\label{eq:OTReDualCon2Sol2Alpha}
\beta_{xy}(k+1) = \beta_{xy}(k) + \frac{\widehat{\eta}}{2} (w_{xy}^{(t)}(k+1) - w_{xy}^{(s)}(k+1)),
\end{equation}
where $\mathcal{U}_x:=\{ u_x,{\mathbf{w}_x^{(t)}} |  \gamma_{xy} - u_x - w_{xy}^{(t)} \leq 0,  y\in\mathcal{Y}  \}$, $\mathcal{U}_y:=\{v_y,{\mathbf{w}_y^{(s)}} |\delta_{xy} - v_y + w_{xy}^{(s)} \leq 0,  x\in\mathcal{X}  \}$.
\end{proposition}
\begin{proof}
See Appendix D.
\end{proof}

\begin{figure}[]
\centering
\subfigure{\includegraphics[width=0.223\textwidth]{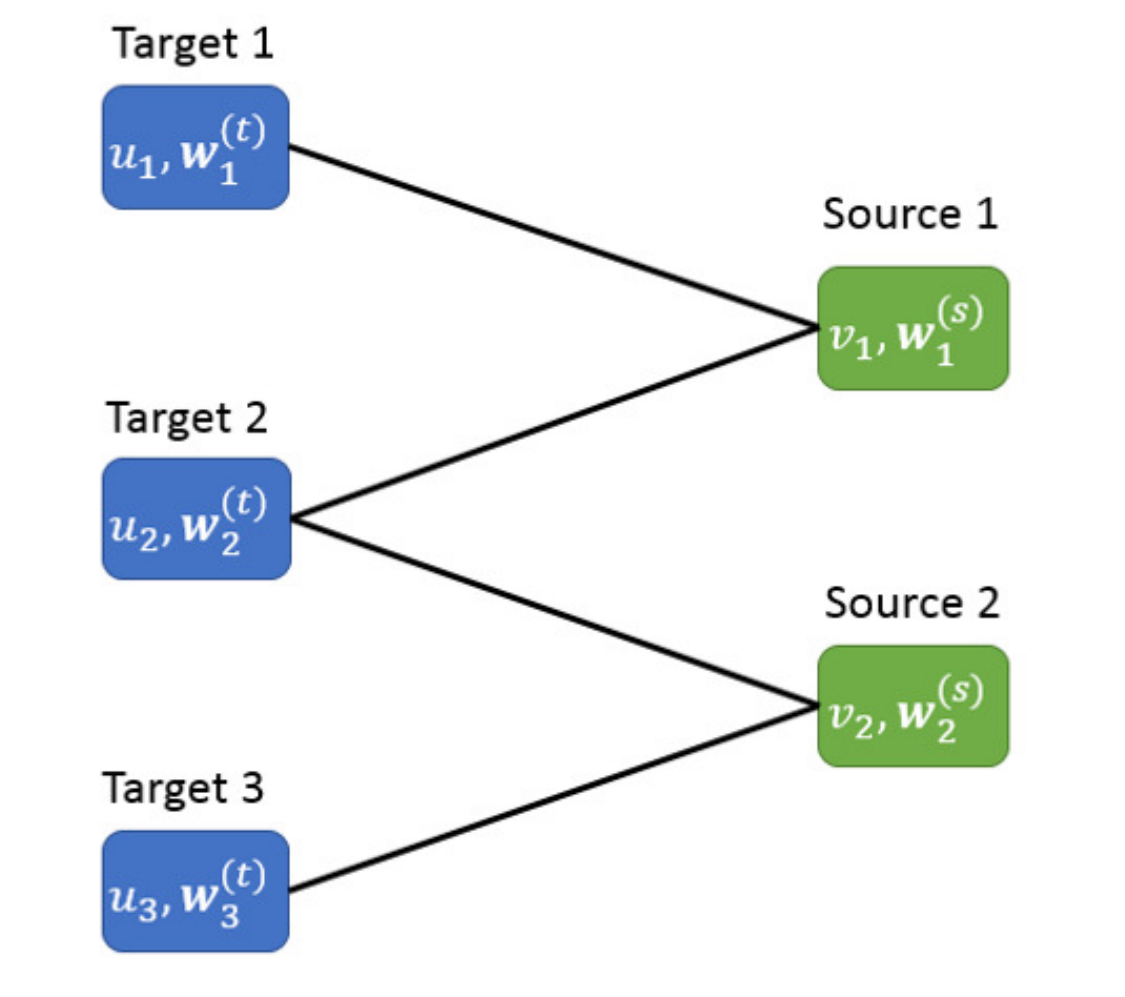}}
\subfigure{\includegraphics[width=0.223\textwidth]{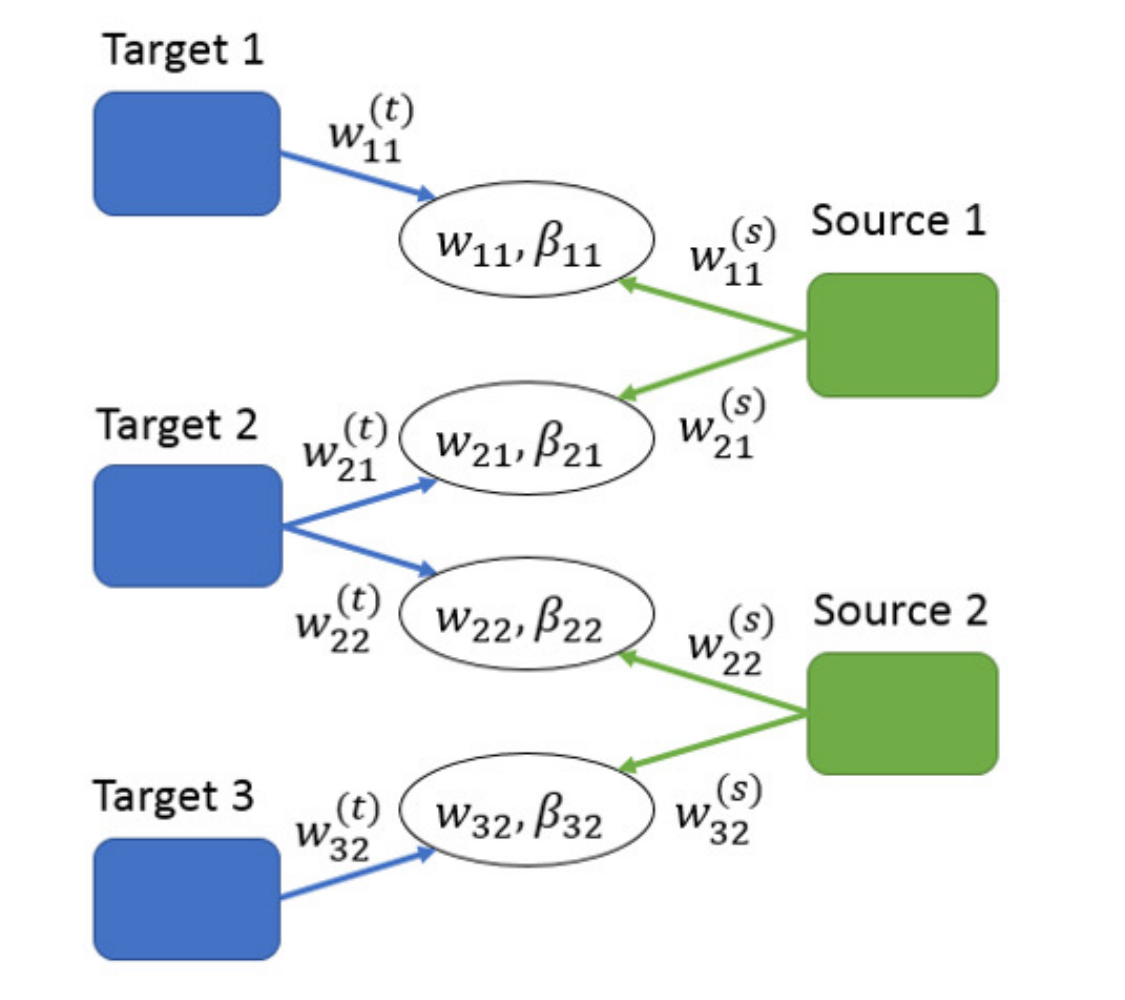}}
\caption{{Illustration of the distributed dual algorithm. At every iteration, each target $x$ or source $y$ achieves $\{u_x,\mathbf{w}_{x}^{(t)}\}$ or $\{v_y,\mathbf{w}_{y}^{(s)}\}$ with (\ref{eq:OTReDualCon2Sol2UW}) or (\ref{eq:OTReDualCon2Sol2VW}), respectively. Then, each pair of target $x$ and source $y$ computes $w_{xy}$ and $\beta_{xy}$ with (\ref{eq:OTReDualCon2Sol2Beta}) and (\ref{eq:OTReDualCon2Sol2Alpha}), respectively. The iterations continue until convergence.}}
\label{fig:Alg4}
\end{figure}

Iterations (\ref{eq:OTReDualCon2Sol2UW})-(\ref{eq:OTReDualCon2Sol2Alpha}) can be used to solve problem (\ref{eq:OTReDualCon}), and they are fully distributed operations. The complexity of the dual algorithm in Proposition \ref{pro:OTRe2} is similar to the complexity of the primal algorithm in Proposition \ref{pro:OTRePrimal2}. Iterations (\ref{eq:OTReDualCon2Sol2UW}) and (\ref{eq:OTReDualCon2Sol2VW}) are individual operations which are quadratic programming with $O(n^3)$ complexity; iterations (\ref{eq:OTReDualCon2Sol2Beta}) and (\ref{eq:OTReDualCon2Sol2Alpha}) are pair operations which are matrix calculations with $O(1)$ complexity. The complexity of the whole algorithm is further affected by the number of participants and connections between targets and sources. The dual algorithm also has similar advantages of efficiency and privacy as the primal algorithm since each participant solves its own problem and their personal information $\{p_x, \mathcal{Y}_x\}$ or $\{q_y, \mathcal{X}_y\}$ is kept private to themselves. However, the information communicated in the dual algorithm is $w_{xy}^{(t)}$, $w_{xy}^{(s)}$, and $w_{xy}$, rather than $\pi_{xy}^{(t)}$, $\pi_{xy}^{(s)}$, and $\pi_{xy}$ in the primal algorithm, which indicates that the participants bargain on the amounts and prices of resources in the primal algorithm and the dual algorithm, respectively. We can also extend the dual algorithm to an online version to address real-time applications following similar methods used in the previous section. The convergence of the dual algorithm is guaranteed from Lemma \ref{lem:ADMoMConvergence} in Appendix A as the objective function $\sum_{x\in\mathcal{X}} u_x p_x + \sum_{y\in\mathcal{Y}} v_y q_y$ is linear. 

The distributed dual algorithm captures the bargaining between sources and targets on the prices ${\mathbf{w}_\mathcal{G}}$ of transferred resources. Target $x$ offers a payment of $w_{xy}^{(t)}$ and source $y$ requests a payment of $w_{xy}^{(s)}$ to transfer resources from $y$ to $x$. The consensus constraints in problem (\ref{eq:OTReDualCon}) capture the agreement between $x$ and $y$ on the prices as $w_{xy}^{(t)} = w_{xy} = w_{xy}^{(s)}$. We can see from (\ref{eq:OTReDualCon2Sol2Beta}) that the agreement is achieved by following an averaging principle that $w_{xy}$ is the average between $w_{xy}^{(t)}$ and $w_{xy}^{(s)}$. Each $x$ or $y$ modifies its proposal $w_{xy}^{(t)}$ or $w_{xy}^{(s)}$ to maximize its own surplus and minimize the differences from its previous negotiations $w_{xy}$ at the same time. $\beta_{xy}$ captures the cumulative differences between $x$ and $y$ on the payments, and it further identifies the directions for $x$ and $y$ to modify their proposals. $\widehat{\eta}$ controls each participant's trade-off between a high surplus of himself and a small difference from others. 

The primal algorithm and the dual algorithm share a lot similarities in terms of their structures and both of them indicates an averaging principle on the bargaining between sources and targets. We further prove the primal-dual relations between their iterations. We shall see that $\alpha_{xy} = w_{xy}$, $\pi_{xy} = \beta_{xy}$, and $\eta = 1/\widehat{\eta}$ where $\alpha_{xy}$, $\pi_{xy}$, and $\eta$ come from the primal algorithm and $w_{xy}$, $\beta_{xy}$, and $\widehat{\eta}$ come from the dual algorithm.

\subsection{Relations Between The Primal Algorithm And The Dual Algorithm}
In this subsection, we prove the primal-dual relations between the iterations of the primal algorithm in Proposition \ref{pro:OTRePrimal2} and the iterations of the dual algorithm in Proposition \ref{pro:OTRe2}.
\begin{proposition}
\label{pro:OTReDualSolDual}
Let {$\lambda_{xy}^{(t)}\in\mathbb{R}$ and $\lambda_{xy}^{(s)}\in\mathbb{R}$} denote the Lagrange multipliers for (\ref{eq:OTReDualCon2Sol2UW}) and (\ref{eq:OTReDualCon2Sol2VW}), respectively. Iterations (\ref{eq:OTReDualCon2Sol2UW})-(\ref{eq:OTReDualCon2Sol2Alpha}) can be rewritten as
\begin{equation}
\label{eq:OTReDualSolDualUW}
\begin{array}{l}
{\Lambda_{x}^{(t)}}(k+1)  \in \arg\min\limits_{{\Lambda_{x}^{(t)}} \in \mathcal{U}_{x}} - \sum\limits_{y\in \mathcal{Y}}  \lambda_{xy}^{(t)} \gamma_{xy} \\ \ \ \ \ \ \ \ \ \ \  + \sum\limits_{y\in\mathcal{Y}} w_{xy}(k)\lambda_{xy}^{(t)}  + \frac{1}{2\widehat{\eta}} \sum\limits_{y\in\mathcal{Y}}  (\lambda_{xy}^{(t)} - \beta_{xy}(k))^2 ,
\end{array}
\end{equation}
\begin{equation}
\label{eq:OTReDualSolDualVW}
\begin{array}{l}
{\Lambda_{y}^{(s)}}(k+1) \in \arg\min\limits_{{\Lambda_{y}^{(s)}} \in \mathcal{U}_{y}} - \sum\limits_{x\in \mathcal{X}}\lambda_{xy}^{(s)} \delta_{xy} \\ \ \ \ \ \ \ \ \ \ \ - \sum\limits_{x\in\mathcal{X}}  w_{xy}(k)\lambda_{xy}^{(s)} + \frac{1}{2\widehat{\eta}} \sum\limits_{x\in\mathcal{X}}  (\beta_{xy}(k)-\lambda_{xy}^{(s)} )^2 ,
\end{array}
\end{equation}
\begin{equation}
\label{eq:OTReDualSolDualBeta}
\begin{array}{l}
\beta_{xy}(k+1)   =\frac{1}{2} (\lambda_{xy}^{(t)}(k+1) + \lambda_{xy}^{(s)}(k+1)) ,
\end{array}
\end{equation}
\begin{equation}
\label{eq:OTReDualSolDualAlpha}
w_{xy}(k+1) = w_{xy}(k) + \frac{1}{2\widehat{\eta}} (\lambda_{xy}^{(t)}(k+1) - \lambda_{xy}^{(s)}(k+1)),
\end{equation}
where $\mathcal{U}_x:=\{ {\Lambda_{x}^{(t)}}| \lambda_{xy}^{(t)} \geq 0, y\in\mathcal{Y}; \sum_{y\in \mathcal{Y}}  \lambda_{xy}^{(t)}  = p_x \}$, $\mathcal{U}_y:=\{{\Lambda_{y}^{(s)}} |\lambda_{xy}^{(s)} \geq 0, x\in\mathcal{X}; \sum_{x\in \mathcal{X}}  \lambda_{xy}^{(s)} = q_y \}$.
\end{proposition}
\begin{proof}
See Appendix E.
\end{proof}
\begin{remark}
\label{rem:OTPD}
The iterations in Proposition \ref{pro:OTRePrimal2} and the iterations in Proposition \ref{pro:OTReDualSolDual} are identical with $\eta = {1}/{\widehat{\eta}}$, $\pi_{xy}^{(t)} = \lambda_{xy}^{(t)}$, $\pi_{xy}^{(s)} = \lambda_{xy}^{(t)}$, $\pi_{xy} = \beta_{xy}$, and $\alpha_{xy} = w_{xy}$.
\end{remark}
Thus, $\alpha_{xy}$ in the primal algorithm captures the prices that target $x$ pays to source $y$ as the compensation of transferring resources, i.e., $w_{xy}$ in the dual problem (\ref{eq:OTReDualCon}). $\beta_{xy}$ in the dual algorithm captures the amount of transferred resources from $y$ to $x$, i.e., $\pi_{xy}$ in the primal problem (\ref{eq:OTReRe}). Both the primal algorithm and the dual algorithm solve problems (\ref{eq:OTReRe}) and (\ref{eq:OTReDualCon}) at the same time with updates on allocations $\pi_{xy}$ and prices $w_{xy}$, and the convergence to the optimum is guaranteed.

\begin{figure}[]
\centering
\subfigure[$k=0$]{\includegraphics[width=0.235\textwidth]{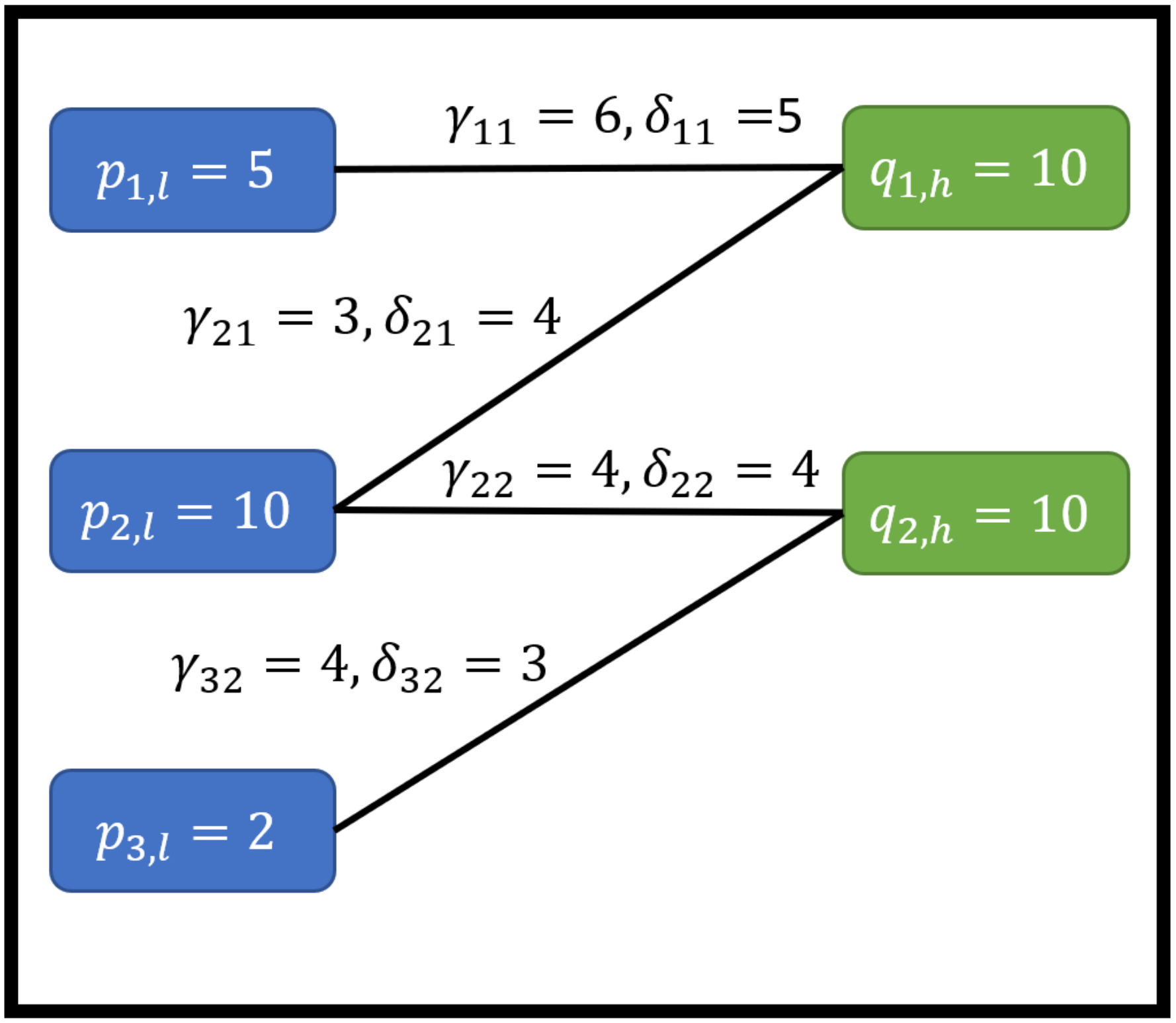}}
\subfigure[$k=400$]{\includegraphics[width=0.235\textwidth]{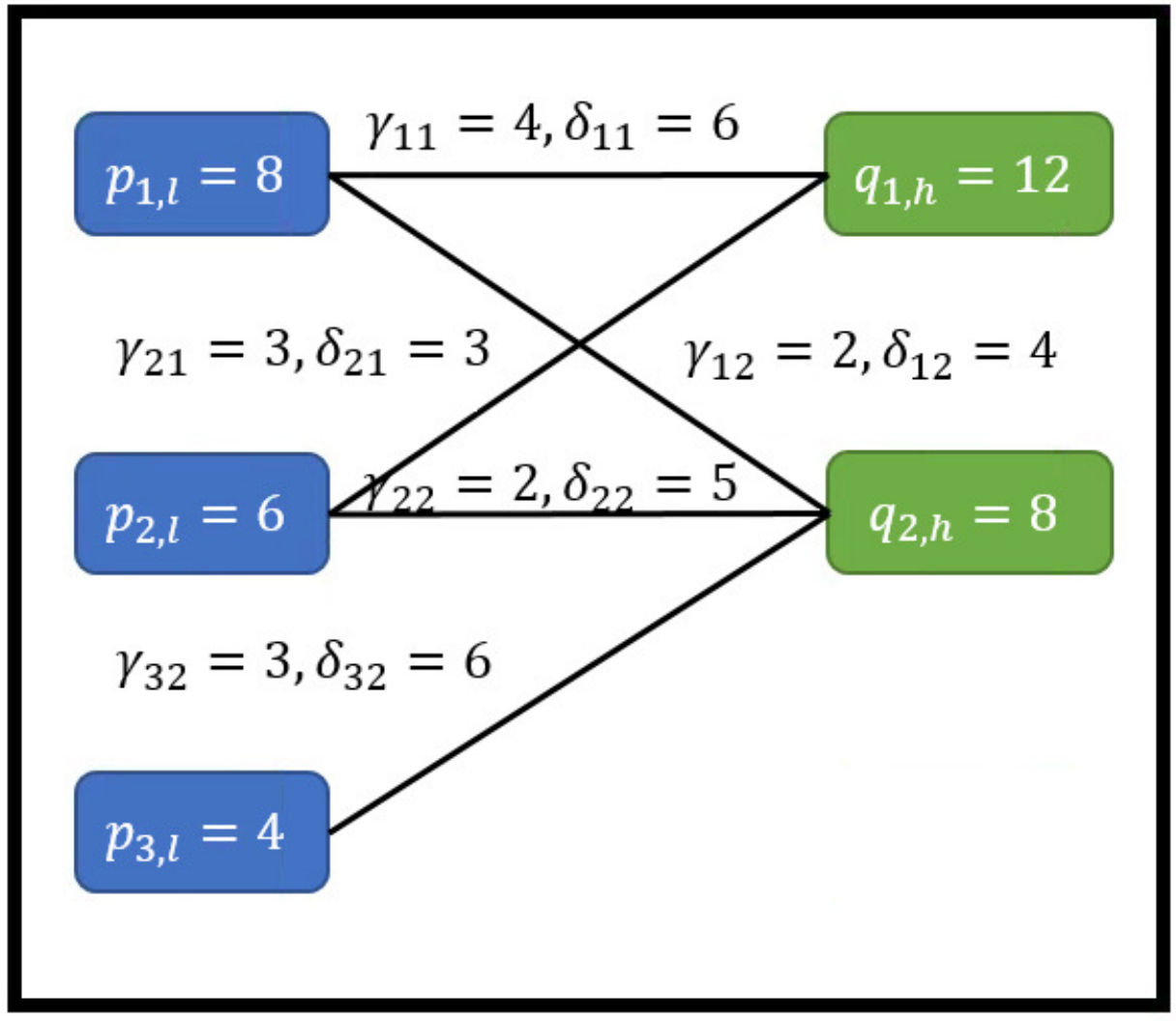}}
\subfigure[$k=800$]{\includegraphics[width=0.235\textwidth]{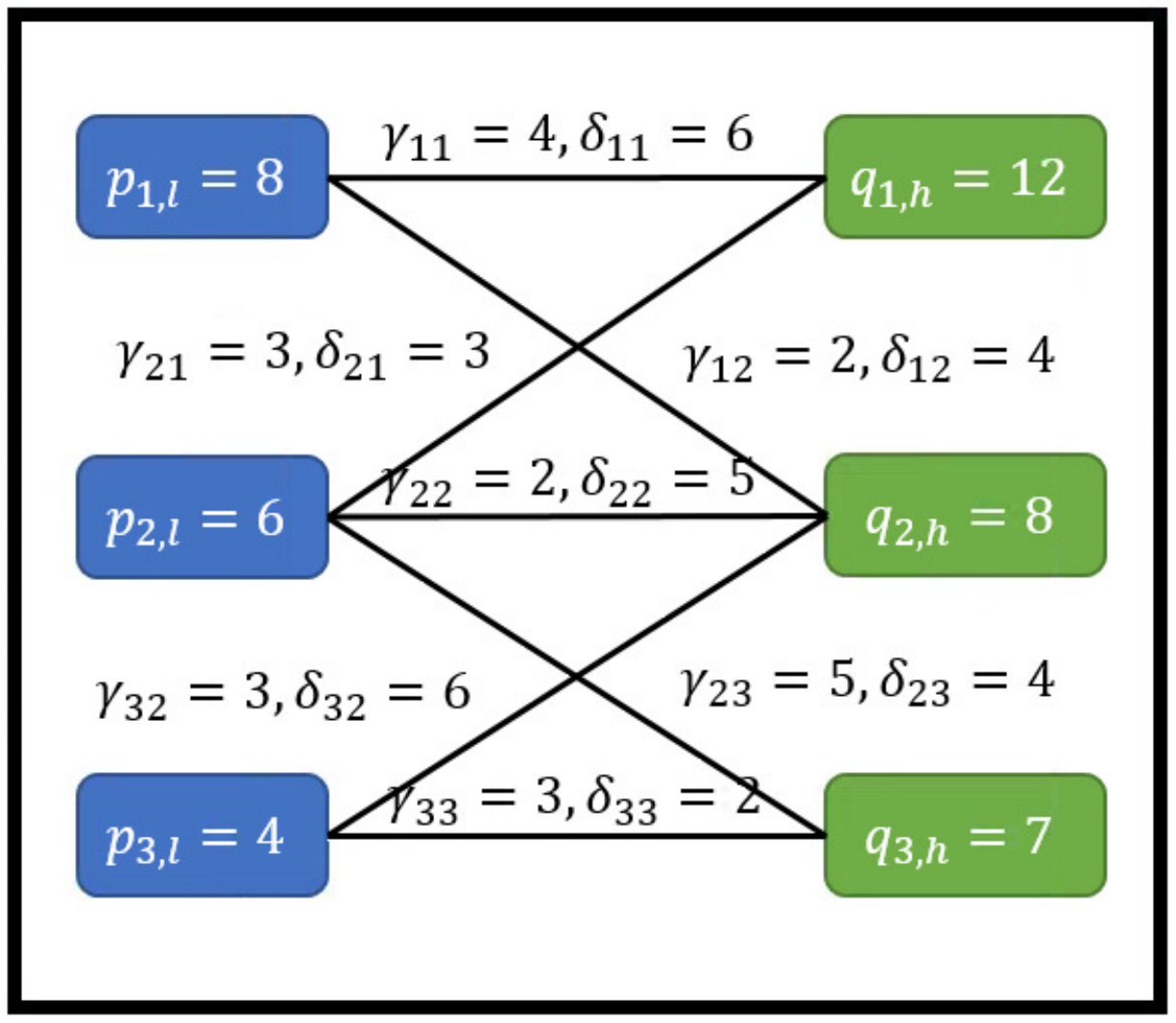}}
\subfigure[$k=1200$]{\includegraphics[width=0.235\textwidth]{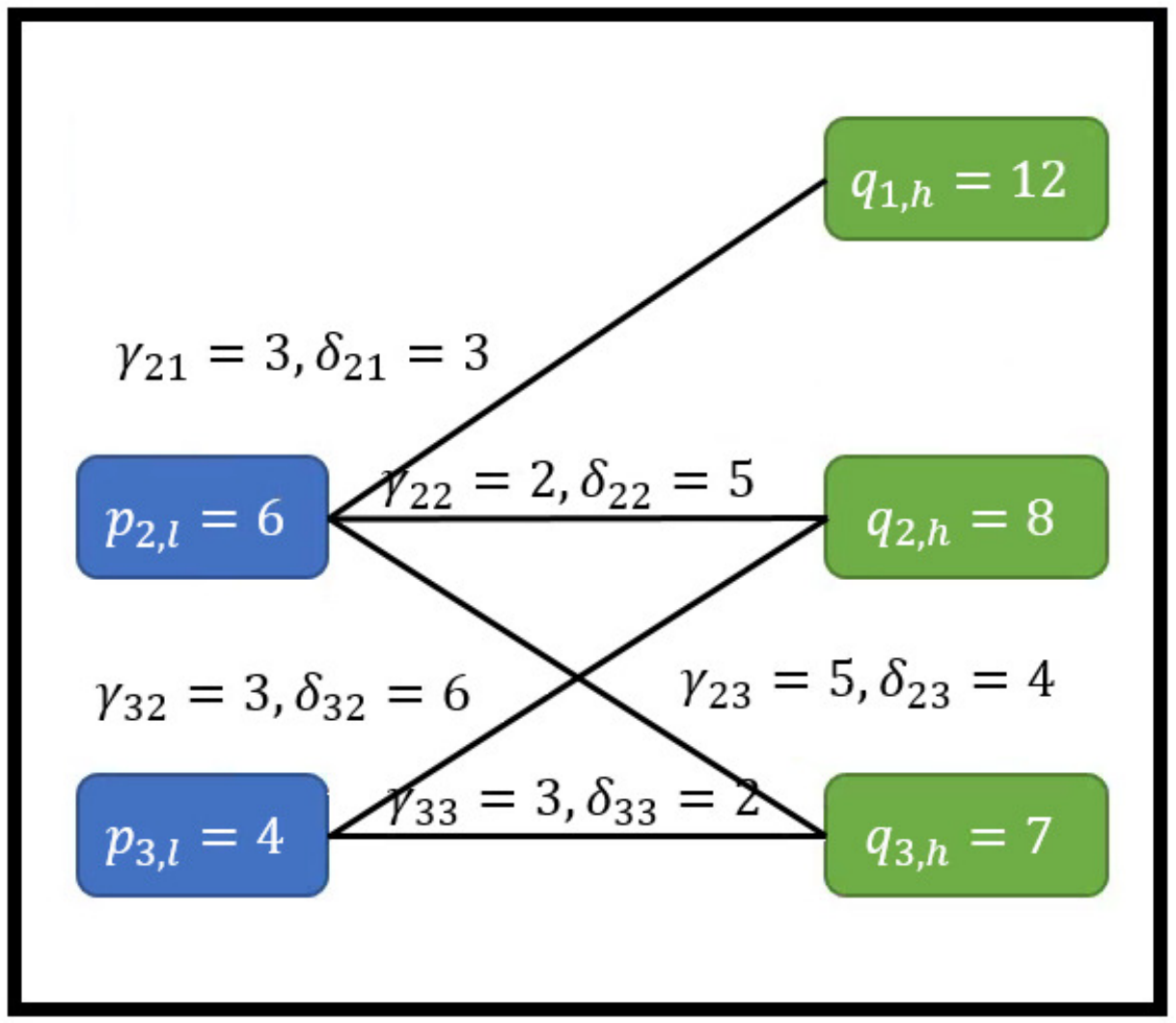}}
\caption{Online Discrete Optimal Transport Problems.  }
\label{fig:EX1Demo}
\end{figure}

\begin{figure}[]
\centering
\subfigure[$D(k) = \sum_{x\in\mathcal{X}}\sum_{y\in\mathcal{Y}_x}f_{xy}(\pi_{xy}(k)) + \sum_{y\in\mathcal{Y}}\sum_{x\in\mathcal{X}_y}g_{xy}(\pi_{xy}(k))$.]{\includegraphics[width=0.49\textwidth]{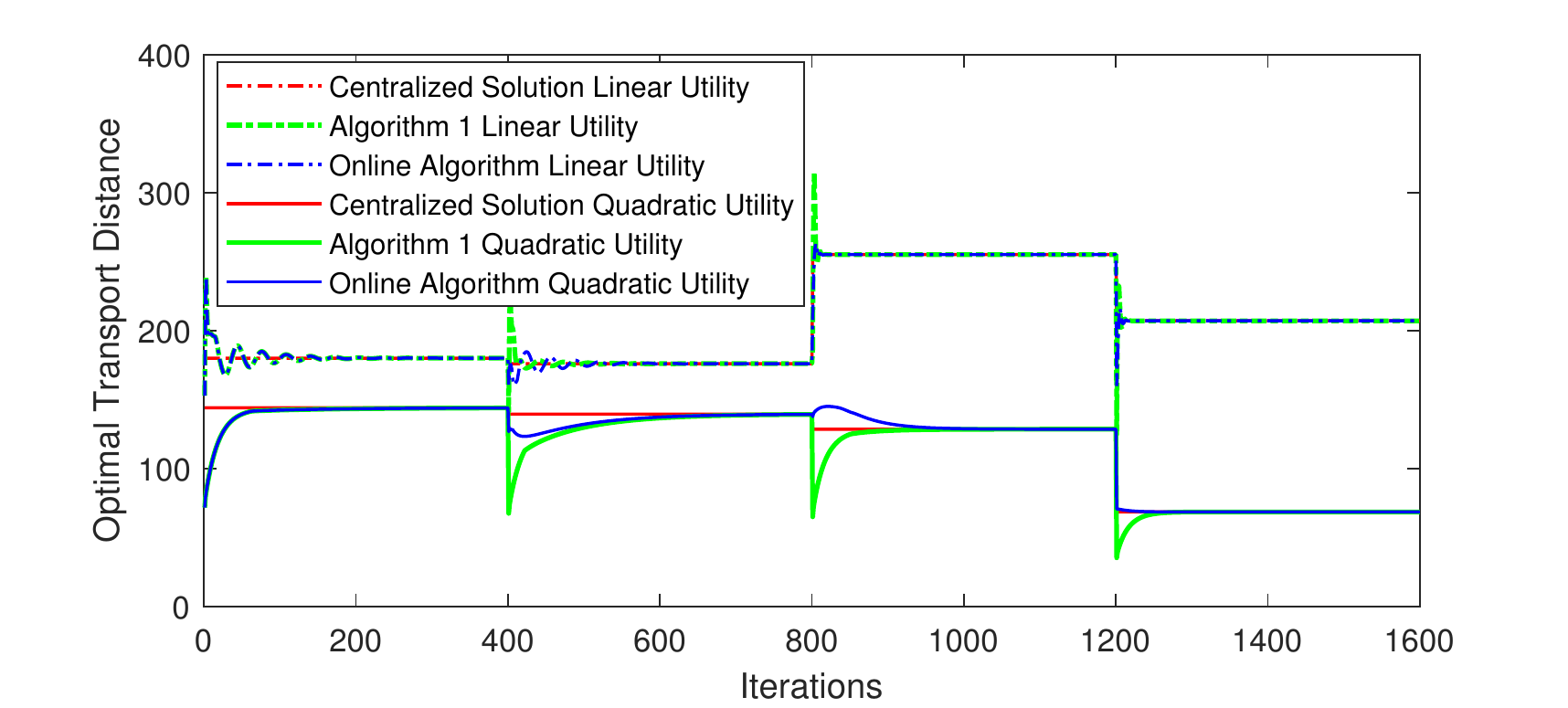}}
\subfigure[$\sqrt{(\text{D}(k) - \text{D}^c)^2}$.]{\includegraphics[width=0.49\textwidth]{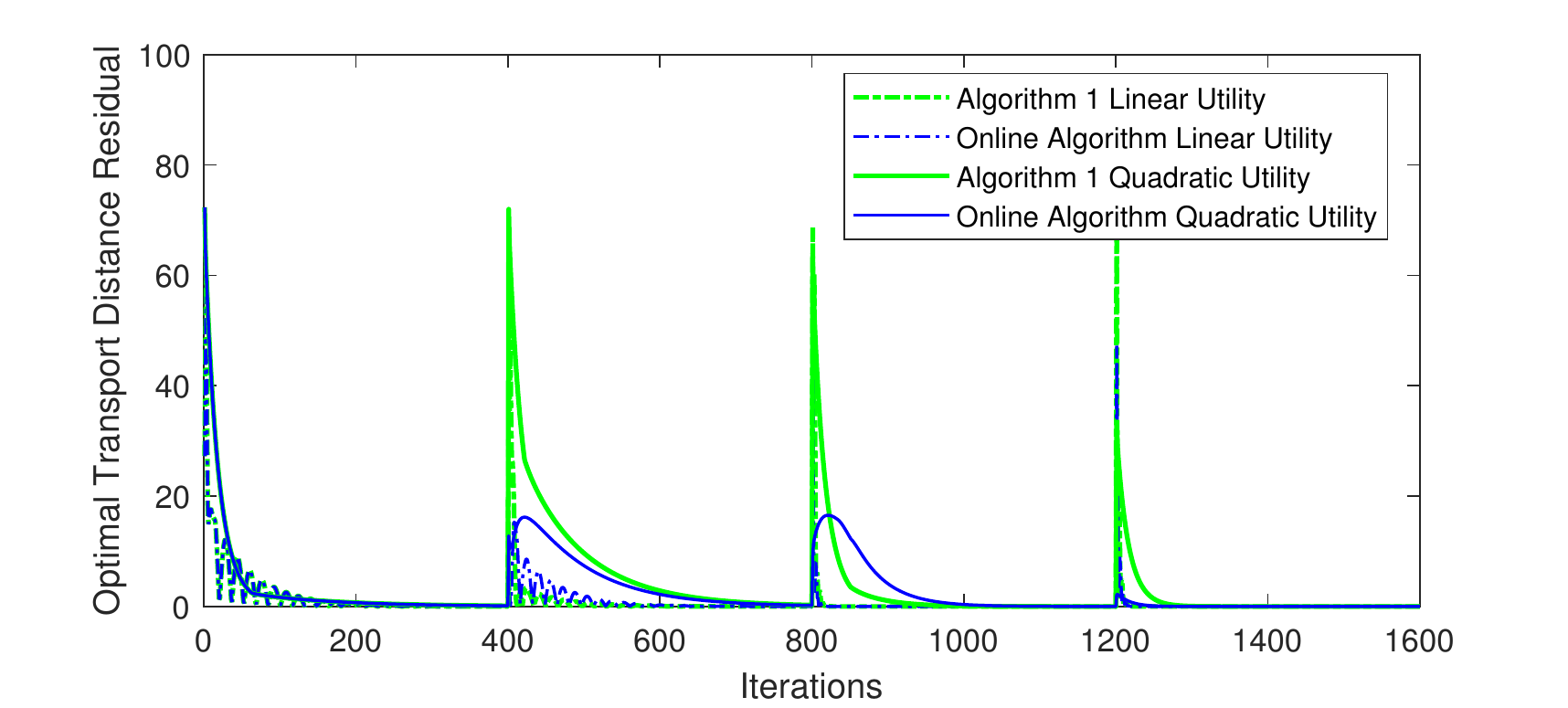}}
\subfigure[$\sqrt{\sum_{x\in\mathcal{X}}\sum_{y\in\mathcal{Y}_x} \left(\pi_{xy}(k) - \pi_{xy}^c \right)^2}$.]{\includegraphics[width=0.49\textwidth]{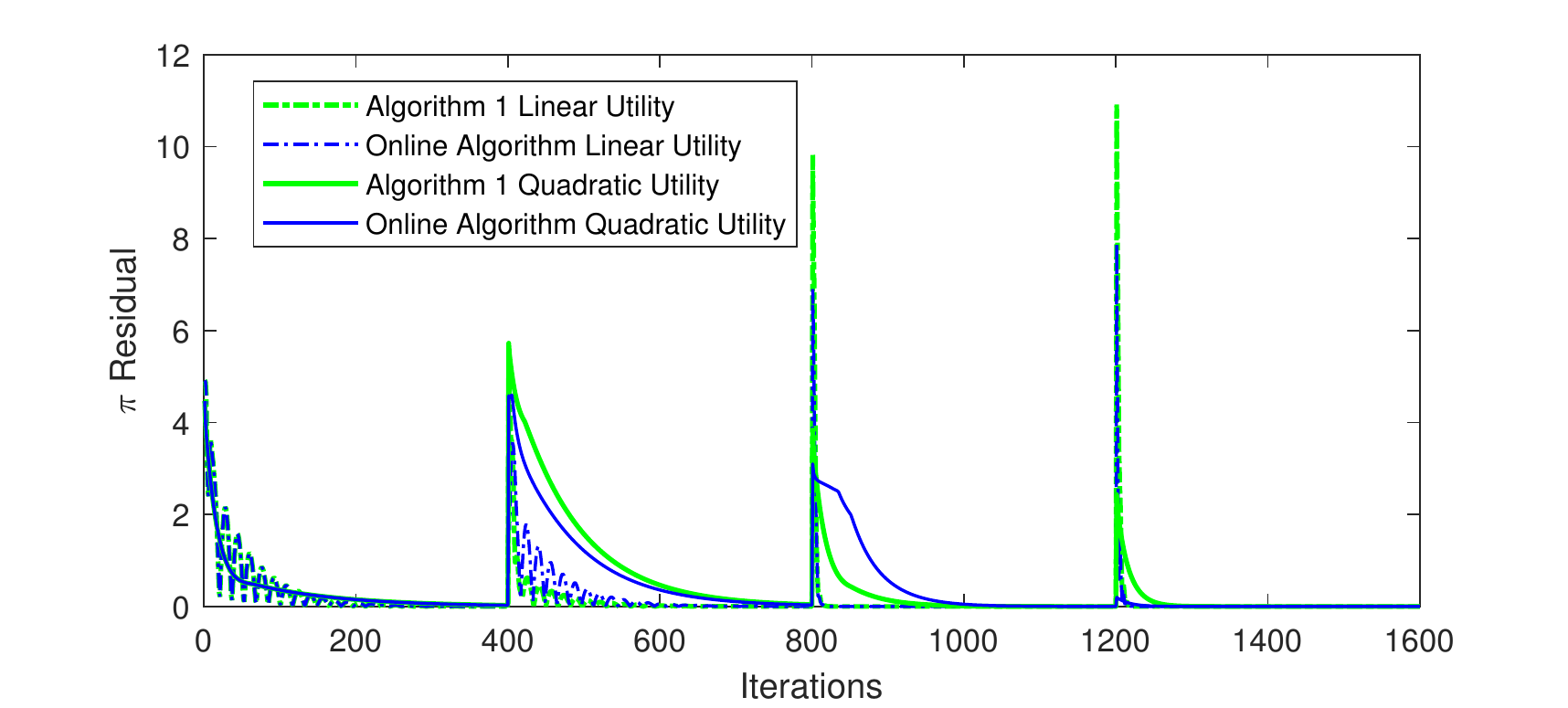}}
\caption{Evolution of Iterations (\ref{eq:MainReReSol2UWOnline})-(\ref{eq:MainReReSol2AlphaOnline}) for the online resource matching problem shown in Fig. \ref{fig:EX1Demo}. The centralized solution ${\Pi_{\mathcal{G}}^c}$ and the corresponding optimal value $D^c$ are achieved by directly solving problem (\ref{eq:Main}). The residuals measure the differences between the current updates and the centralized solution. }
\label{fig:EX1}
\end{figure}

\section{Numerical Experiments}
\label{sec:NE}
In this section, we present numerical experiments of distributed discrete optimal transport. We first verify the convergence of Algorithm 1 to the solution of problem (\ref{eq:Main}), and we also demonstrate the online adaptability of the online algorithm presented in Remark \ref{rem:MainReReOnline}. The experiment settings are described in Fig. \ref{fig:EX1Demo}. Blue rectangles represent targets while green rectangles represent sources. We let all $p_{x, h} = 100$ which indicates that targets have high upper bounds on acquiring resources, and all $q_{y, l} = 0$ which indicates that sources can choose not to allocate any resources. We consider both linear and quadratic utility functions with $\gamma_{xy}$ and $\delta_{xy}$ denoting the parameters, i.e, $f_{xy}^l(\pi_{xy}^{(t)}) = C_{xy}^{(t)} + \gamma_{xy}\pi_{xy}^{(t)}$, $g_{xy}^l(\pi_{xy}^{(t)}) = C_{xy}^{(s)} + \delta_{xy}\pi_{xy}^{(s)}$, $f_{xy}^q(\pi_{xy}^{(t)}) = C_{xy}^{(t)} - \gamma_{xy}(\pi_{xy}^{(t)})^2$, and $g_{xy}^q(\pi_{xy}^{(s)}) = C_{xy}^{(s)} - \delta_{xy}(\pi_{xy}^{(s)})^2$. $ C_{xy}^{(t)}$ and $C_{xy}^{(s)}$ are constant to $\pi_{xy}^{(t)}$ and $\pi_{xy}^{(s)}$. When $k=0$, the resource matching market has 3 targets and 2 sources. When $k=400$, Target 1 and Target 3 increase their demands while Target 2 decreases its demand; Source 1 increases its supplies while Source 2 decreases its supplies; Target 1 builds connection to Source 2; some participants also have different utility parameters $\gamma_{xy}$ or $\delta_{xy}$. When $k=800$, Source 3 joins the market and provides resources to Target 2 and Target 3. When $k=1200$, Target 1 quits and the market now has two targets and three sources. All four cases are feasible which can be easily verified. Algorithm 1 solves each case separately, while the online algorithm solves all cases continuously. The results of them are shown in Fig. \ref{fig:EX1}. We can see from Fig. \ref{fig:EX1} that both Algorithm 1 and the online algorithm converge to the centralized optimal solutions ${\Pi_{\mathcal{G}}^c}$ and the centralized optimal values $D^c$ in all cases. Since we do not need to restart the online algorithm when the problem changes, the online algorithm is suitable for online resource matching problems. 

\begin{figure}[http]
\centering
\subfigure[$\text{D}(k) = \sum_{x\in\mathcal{X}}\sum_{y\in\mathcal{Y}}\pi_{xy}(k)\gamma_{xy} + \sum_{y\in\mathcal{Y}}\sum_{x\in\mathcal{X}}\pi_{xy}(k)\delta_{xy}$.]{\includegraphics[width=0.49\textwidth]{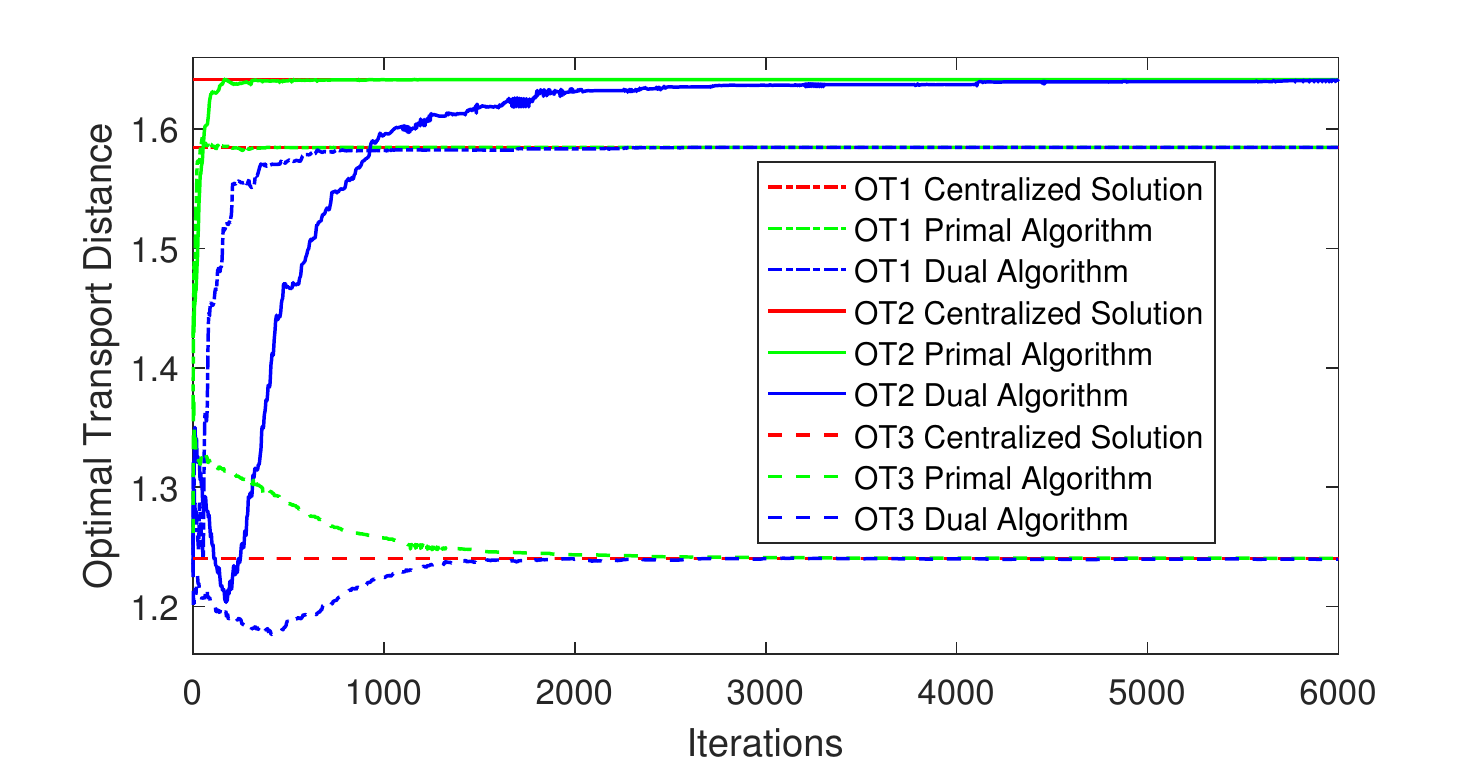}}
\subfigure[$\sqrt{(\text{D}(k) - \text{D}^c)^2}$.]{\includegraphics[width=0.49\textwidth]{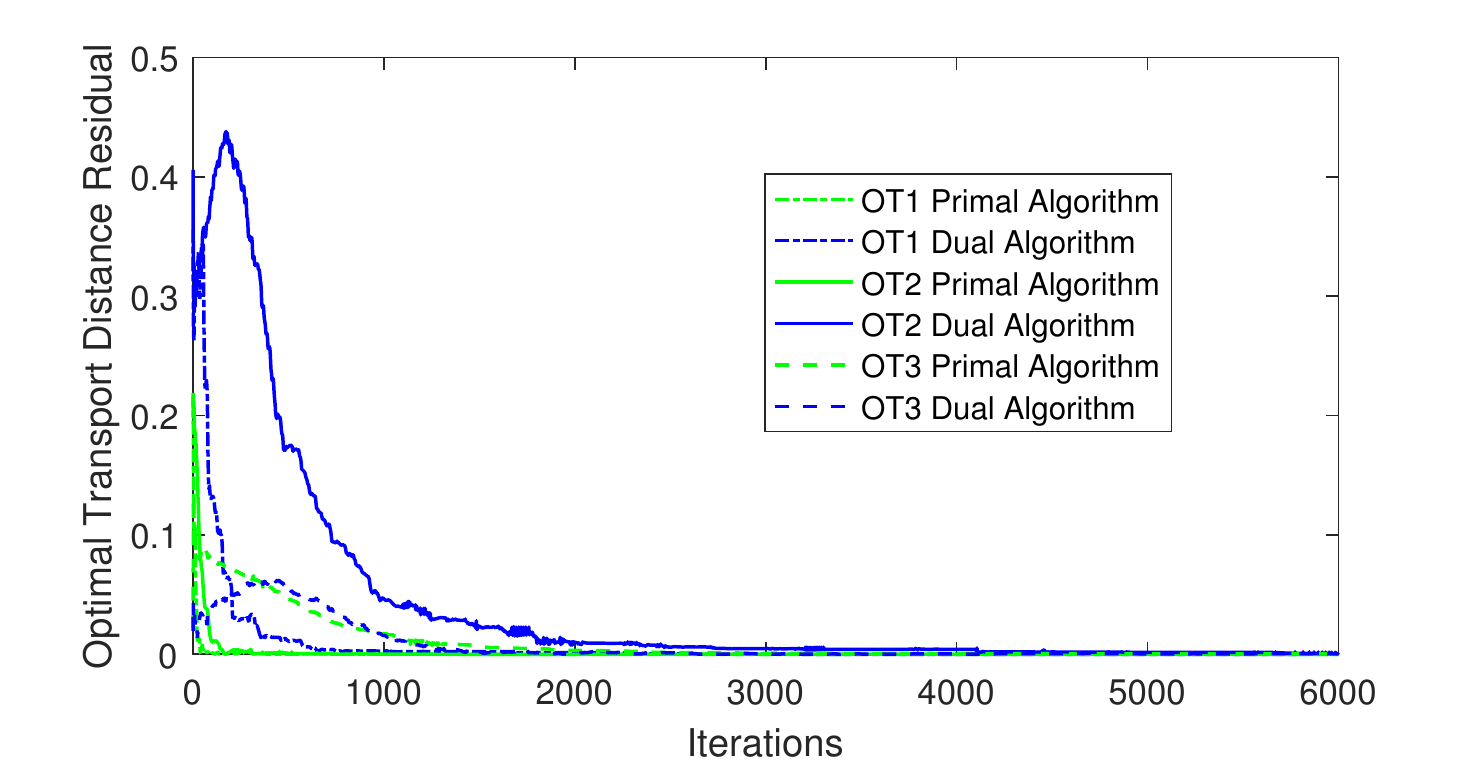}}
\subfigure[$\sqrt{\sum_{x\in\mathcal{X}}\sum_{y\in\mathcal{Y}} \left(\pi_{xy}(k) - \pi_{xy}^c \right)^2}$.]{\includegraphics[width=0.49\textwidth]{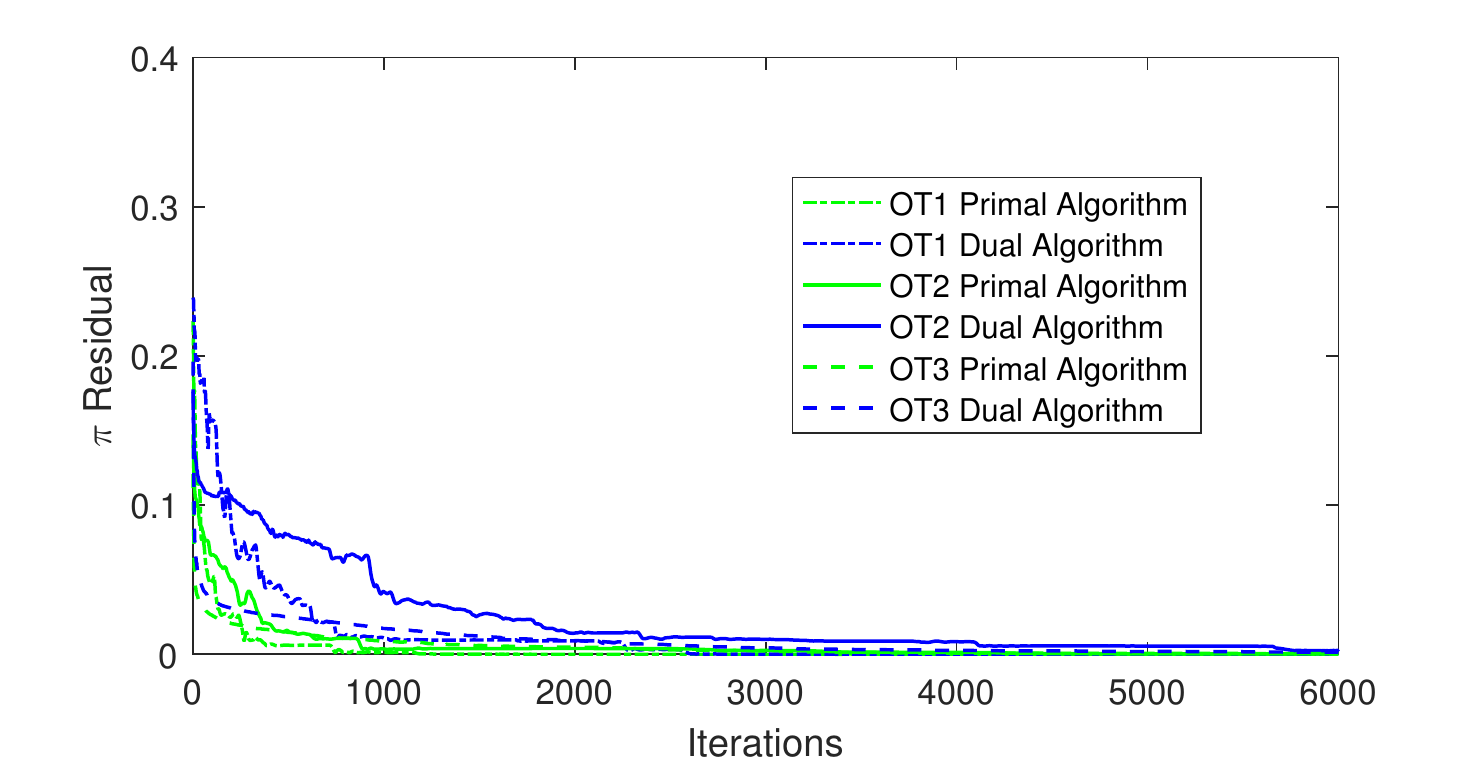}}
\caption{Evolution of the primal algorithm and the dual algorithm. The centralized solution is achieved by directly solving problem (\ref{eq:OptimalTransport}). The residuals measure the differences between the current updates and the centralized solution. }
\label{fig:EX0}
\end{figure}

We then show that the primal algorithm in Proposition \ref{pro:OTRePrimal2} and the dual algorithm in Proposition \ref{pro:OTRe2} can be used to solve discrete optimal transport problems (\ref{eq:OptimalTransport}) with linear utility functions, and the convergence to the solutions is guaranteed. Let $p_x$ and $q_y$ be generated by i.i.d. uniform distributions on $(0,1)$. We then normalize $\{p_x\}$ and $\{q_y\}$ to $1$ so that $\sum_{x\in\mathcal{X}} p_x = 1 = \sum_{y\in\mathcal{Y}} q_y$ as required in Condition \ref{con:exisitanceOT} which guarantees the existence of a solution. $\gamma_{xy}$ and $\delta_{xy}$ are also generated by i.i.d. uniform distributions on $(0,1)$. We consider three discrete optimal transport problems `OT1', `OT2', and `OT3' with different numbers of sources and targets under this setting. `OT1' has 20 targets and 20 sources; `OT2' has 20 targets and 100 sources; `OT3' has 1000 targets and 2 sources. The results are shown in Figure \ref{fig:EX0}. We can see that both the primal algorithm and the dual algorithm converge to the centralized solution ${\Pi_{\mathcal{G}}^c}$ and the centralized optimal value $D^c$ of problem (\ref{eq:OptimalTransport}).

\section{Conclusion}
\label{sec:CON}
In this paper, we have studied a decentralized discrete optimal transport problem using a consensus-based approach and have derived distributed algorithms with ADMM to address resource matching between multiple sources and multiple targets. We have shown that our algorithms are fully distributed and do not require central planners to compute the optimal matching or store personal information of sources and targets, which provides guaranteed levels of efficiency and privacy at the same time. We have developed distributed primal and dual algorithms by leveraging the primal and dual problems of discrete optimal transport with linear utility functions and further proved the equivalence between them. We have demonstrated with numerical experiments the convergence of our algorithms to centralized solutions, the online adaptability of our algorithms for real-time applications, and the equivalence between the primal algorithm and the dual algorithm for cases with linear utility functions. We have shown that our algorithms offer useful insights of resource matching in market environments as the interactions between sources and targets reflect the bargaining between them on the amount and price of transferred resources. Furthermore, our algorithms have revealed an averaging principle that sources and targets follow during their bargaining process, which can be used to regulate resource markets and improve resource efficiency. One future direction is to deploy our algorithms to various applications and explore other decentralization methods. {Another future direction is to investigate private resource matching through differential privacy. }

{
\section*{Appendix A: Alternating Direction Method of Multipliers}
Alternating Direction Method of Multipliers (ADMM) is an algorithm for solving problems in the following form \cite{boyd2011distributed}:
\begin{equation}
\label{eq:ADMoM}
\begin{array}{l}
\min\limits_{\mathbf{x}\in \mathcal{U}_\mathbf{x},\mathbf{y}\in\mathcal{U}_\mathbf{y}} f(\mathbf{x}) + g(\mathbf{y})\\
{\rm{s}}{\rm{.t}}{\rm{.  \ \  }}A\mathbf{x} = \mathbf{y}.
\end{array}
\end{equation}
where $f:\mathbb{R}^m \rightarrow \mathbb{R}\cup\{+\infty\}$ and $g:\mathbb{R}^n \rightarrow \mathbb{R}\cup\{+\infty\}$ are convex functions, $A\in \mathbb{R}^{n\times m}$ is a linear operator. The optimal value is denoted by
\[{p^*} = \inf\limits_{\mathbf{x}\in \mathcal{U}_\mathbf{x},\mathbf{y}\in\mathcal{U}_\mathbf{y}} \{ f(\mathbf{x}) + g(\mathbf{y})\left| {A\mathbf{x} = \mathbf{y}} \right.\}. \]
With $\alpha \in \mathbb{R}^n$ donating the Lagrange multiplier corresponding to the constraint $A\mathbf{x} = \mathbf{y}$, the augmented Lagrangian is 
\begin{equation}
\label{eq:ADMoMLagrange}
\mathcal{L}(\mathbf{x}, \mathbf{y},\alpha ) = f(\mathbf{x}) + g(\mathbf{y}) + {\alpha ^T}(A\mathbf{x} - \mathbf{y}) + \frac{\eta }{2}\left\| {A\mathbf{x} - \mathbf{y}} \right\|_2^2,
\end{equation}
where the parameter $\eta>0$ controls the impact of the constraint violation in (\ref{eq:ADMoM}). The ADMM iterations are given by
\begin{equation}
\label{eq:ADMoMSolX}
\mathbf{x}(k + 1) \in \arg {\min _{\mathbf{x}\in \mathcal{U}_\mathbf{x}}} \mathcal{L}(\mathbf{x},\mathbf{y}(k),\alpha (k));
\end{equation} 
\begin{equation}
\label{eq:ADMoMSolY}
\mathbf{y}(k + 1) \in \arg {\min _{\mathbf{y}\in\mathcal{U}_\mathbf{y}}}\mathcal{L}(\mathbf{x}(k + 1),\mathbf{y},\alpha (k));
\end{equation}
\begin{equation}
\label{eq:ADMoMSolAlpha}
\alpha (k + 1) = \alpha (k) + \eta (A\mathbf{x}(k + 1) -\mathbf{y}(k+1)).
\end{equation}
Furthermore, we have the following lemma regarding the convergence of ADMM from Section 3.2 in \cite{boyd2011distributed}. 
\begin{lemma}
\label{lem:ADMoMConvergence}
Under the following assumptions:
\begin{itemize}
\item[(i)] $f:\mathbb{R}^m \rightarrow \mathbb{R}\cup\{+\infty\}$ and $g:\mathbb{R}^n \rightarrow \mathbb{R}\cup\{+\infty\}$ are closed, proper, and convex;
\item[(ii)] The unaugmented Lagrangian $\mathcal{L}_0(\mathbf{x},\mathbf{y},\alpha ) = f(\mathbf{x}) + g(\mathbf{y}) + {\alpha ^T }(A\mathbf{x}- \mathbf{y})$ has a saddle point, i.e., there exist $(\mathbf{x}^*,\mathbf{y}^*,\alpha^*)$, not necessarily unique, for which
$\mathcal{L}_0(\mathbf{x}^*,\mathbf{y}^*,\alpha) \le \mathcal{L}_0(\mathbf{x}^*,\mathbf{y}^*,\alpha^*) \le \mathcal{L}_0(\mathbf{x},\mathbf{y},\alpha^*)$ holds for all $\mathbf{x},\mathbf{y},\alpha$,
\end{itemize}
iterations (\ref{eq:ADMoMSolX})-(\ref{eq:ADMoMSolAlpha})  converge as follows:
\begin{itemize}
\item Residual convergence. $r(k)=  ||A\mathbf{x}(k) - \mathbf{y}(k) ||_2 \rightarrow  0$ as $k\rightarrow \infty$, i.e., the iterations approach feasibility.
\item Objective convergence. $f(\mathbf{x}(k))+g(\mathbf{y}(k)))\rightarrow p^*$ as $k \rightarrow \infty $, i.e., the objective function of the iterations approaches the optimal value.
\item Dual variable convergence. $\alpha(k)\rightarrow \alpha^*$ as $k\rightarrow \infty$, where $\alpha^*$ is a dual optimal point.
\end{itemize}
\end{lemma}}

{
\section*{Appendix B: Proof of Proposition \ref{pro:MainReRe}}
\label{sec:ADMM}
We first need to write problem (\ref{eq:MainReRe}) into the form of (\ref{eq:ADMoM}). Let us define $\mathbf{x} = [vec({\Pi_{\mathcal{G}}^{(t)}})^T,vec({\Pi_{\mathcal{G}}})^T]^T$, $\mathbf{y} = [vec({\Pi_{\mathcal{G}}})^T,vec({\Pi_{\mathcal{G}}^{(s)}})^T]^T$, and $\alpha = [vec(\{\alpha_{xy1}\})^T,vec(\{\alpha_{xy2}\})^T]^T$, where $vec$ indicates vectorization. Note that $\mathbf{x}$, $\mathbf{y}$, and $\alpha$ are all $2G\times 1$ column vectors, where $G$ is the number of connections between targets and sources, i.e., the size of $\mathcal{G}$. Thus, the consensus constraints can be rewritten into its matrix form as $A\mathbf{x}=\mathbf{y}$, where $A=[\mathbf{I}_{G,G},\mathbf{0}_{G,G};\mathbf{0}_{G,G},\mathbf{I}_{G,G}]$. Note that $\mathbf{0}_{G,G}$ and $\mathbf{I}_{G,G}$ denote a zero square matrix and an identity matrix of dimension $G$, respectively. Moreover, we have $\mathbf{x}\in\mathcal{U}_{\mathbf{x}}$ and $\mathbf{y}\in\mathcal{U}_{\mathbf{y}}$, where $\mathcal{U}_\mathbf{x}=\left\lbrace \mathbf{x}  | \pi_{xy}^{(t)} \geq 0,  p_{x, l} \leq \sum_{y\in \mathcal{Y}_x}  \pi_{xy}^{(t)}  \leq p_{x, h},  \{x, y\} \in \mathcal{G} \right\rbrace$ and $\mathcal{U}_\mathbf{y}=\left\lbrace \mathbf{y} |\pi_{xy}^{(s)} \geq 0, q_{y, l} \leq \sum_{x\in \mathcal{X}_y}  \pi_{xy}^{(s)} \leq q_{y, h},  \{x, y\} \in \mathcal{G} \right\rbrace$. As a result, we can use iterations (\ref{eq:ADMoMSolX})-(\ref{eq:ADMoMSolAlpha}) to solve problem (\ref{eq:MainReRe}). Since there is no coupling among ${\Pi_{x}^{(t)}}$, ${\Pi_{y}^{(s)}}$, $\pi_{xy}$, $\alpha_{xy1}$, and $\alpha_{xy2}$, we can decompose (\ref{eq:ADMoMSolX})-(\ref{eq:ADMoMSolAlpha}) as (\ref{eq:MainReReSolUW})-(\ref{eq:MainReReSolAlpha2}).}

\section*{Appendix C: Proof of Theorem \ref{the:OTReDual}}
Note that problem (\ref{eq:OTReRe}) is a minimization problem with the objective function $-\sum_{x\in\mathcal{X}} \sum_{y\in \mathcal{Y}} \pi_{xy}^{(t)}\gamma_{xy} - \sum_{y\in\mathcal{Y}  }  \sum_{x\in \mathcal{X}}\pi_{xy}^{(s)}\delta_{xy}$. We consider problem (\ref{eq:OTReRe}) as maximizing over $\sum_{x\in\mathcal{X}} \sum_{y\in \mathcal{Y}} \pi_{xy}^{(t)}\gamma_{xy} + \sum_{y\in\mathcal{Y}  }  \sum_{x\in \mathcal{X}}\pi_{xy}^{(s)}\delta_{xy}$ in this proof. Let $u_x$, $v_y$, $w_{xy}^{(t)}$, $w_{xy}^{(s)}$, $\epsilon_{xy}^{(t)}$, and $\epsilon_{xy}^{(s)}$ denote the Lagrange multipliers with respect to $\sum_{y\in\mathcal{Y}} \pi_{xy}^{(t)}= p_x$, $\sum_{x \in\mathcal{X}} \pi_{xy}^{(s)} = q_y$, $\pi_{xy}^{(t)} = \pi_{xy}$, $\pi_{xy} = \pi_{xy}^{(s)}$, $\pi_{xy}^{(t)}\geq 0$, and $\pi_{xy}^{(s)}\geq 0$, respectively, we can achieve the dual problem of problem  (\ref{eq:OTReRe}) as \cite{dantzig1998linear}
\begin{equation}
\label{eq:OTReDualOld}
\begin{array}{c}
\min\limits_{\{{\mathbf{u}_\mathcal{X}},{\mathbf{v}_{\mathcal{Y}}},{\mathbf{w}_{\mathcal{G}}^{(t)}},{\mathbf{w}_{\mathcal{G}}^{(s)}},{\epsilon_{\mathcal{G}}^{(t)}},{\epsilon_{\mathcal{G}}^{(s)}}\}}  \sum\limits_{x\in\mathcal{X}}u_x p_x  + \sum\limits_{y\in\mathcal{Y}} v_y q_y\\
\begin{array}{cc}
{\begin{array}{c}
	\text{s.t.}\\ \ 
	\end{array}}&{ \begin{array}{cc}
	{ \gamma_{xy} - u_x - w_{xy}^{(t)} + \epsilon_{xy}^{(t)}  =  0,}&{\forall \{x,y\}\in\mathcal{G};}\\
	{\delta_{xy}  - v_y  + w_{xy}^{(s)} + \epsilon_{xy}^{(s)} = 0,}&{\forall \{x,y\}\in\mathcal{G};} \\ {w_{xy}^{(t)} = w_{xy}^{(s)} ,}&{\forall \{x,y\}\in\mathcal{G};} \\ {\epsilon_{xy}^{(t)} \geq 0, \epsilon_{xy}^{(s)} \geq 0,}&{\forall \{x,y\}\in\mathcal{G};}
	\end{array}}
\end{array}
\end{array}
\end{equation}
By introducing $\epsilon_{xy}^{(t)} \geq 0$ and $\epsilon_{xy}^{(s)} \geq 0$ into $\gamma_{xy} - u_x - w_{xy}^{(t)} + \epsilon_{xy}^{(t)}  =  0$ and $\delta_{xy}  - v_y  + w_{xy}^{(s)} + \epsilon_{xy}^{(s)} = 0$, respectively, we can achieve problem (\ref{eq:OTReDual}).

\section*{Appendix D: Proof of Proposition \ref{pro:OTRe2}}
We follow similar steps in the proofs of Proposition \ref{pro:MainReRe} and Proposition \ref{pro:MainReRe2}. Let $\mathbf{x} = [{\mathbf{u}_\mathcal{X}}^T,{\mathbf{v}_\mathcal{Y}}^T,vec({\mathbf{w}_\mathcal{G}^{(t)}})^T,vec({\mathbf{w}_\mathcal{G}})^T]^T$ and $\mathbf{y} = [vec({\mathbf{w}_\mathcal{G}})^T,vec({\mathbf{w}_\mathcal{G}^{(s)}})^T]^T$. Thus, the consensus constraints can be captured as $\mathbf{Ax} = \mathbf{y}$, where $A=[\mathbf{0}_{NM,N+M},\mathbf{I}_{NM,NM},\mathbf{0}_{NM,NM};\mathbf{0}_{NM,N+M},\mathbf{0}_{NM,NM},\mathbf{I}_{NM,NM}]$. Note that $N$ and $M$ are the numbers of targets and sources, respectively. As a result, we can achieve the following iterations to solve (\ref{eq:OTReDualCon}). 
\begin{equation}
\label{eq:OTReDualCon2SolUW}
\begin{array}{l}
\left\{u_x(k+1),{\mathbf{w}_x^{(t)}}(k+1)\right\}   \in \arg\min\limits_{\left\{u_x,{\mathbf{w}_x^{(t)}}\right\} \in \mathcal{U}_{x}}   u_xp_x \\ \ \ \ \ \ \ \ \ \ \ + \sum\limits_{y\in\mathcal{Y}} \beta_{xy1}(k)w_{xy}^{(t)}   + \frac{\widehat{\eta}}{2} \sum\limits_{y\in\mathcal{Y}}  (w_{xy}^{(t)} - w_{xy}(k))^2 ,
\end{array}
\end{equation}
\begin{equation}
\label{eq:OTReDualCon2SolVW}
\begin{array}{l}
\left\{v_y(k+1),{\mathbf{w}_y^{(s)}}(k+1) \right\}  \in \arg\min\limits_{\left\{v_y,{\mathbf{w}_y^{(s)}} \right\}  \in \mathcal{U}_{y}}  v_yq_y \\ \ \ \ \ \ \ \ \ \ \ - \sum\limits_{x\in\mathcal{X}}  \beta_{xy2}(k)w_{xy}^{(s)}  + \frac{\widehat{\eta}}{2} \sum\limits_{x\in\mathcal{X}}  (w_{xy}(k)-w_{xy}^{(s)} )^2 ,
\end{array}
\end{equation}
\begin{equation}
\label{eq:OTReDualCon2SolBeta}
\begin{array}{l}
w_{xy}(k+1) \in \arg\min\limits_{w_{xy}} -\beta_{xy1}(k) w_{xy} + \beta_{xy2}(k)w_{xy} \\  \ \ \ \ \ \ \ \  + \frac{\widehat{\eta}}{2}(w_{xy}^{(t)}(k+1)-w_{xy})^2 +\frac{\widehat{\eta}}{2}(w_{xy}-w_{xy}^{(s)}(k+1))^2 ,
\end{array}
\end{equation}
\begin{equation}
\label{eq:OTReDualCon2SolAlpha1}
\beta_{xy1}(k+1) = \beta_{xy1}(k) + \widehat{\eta} (w_{xy}^{(t)}(k+1) - w_{xy}(k+1)),
\end{equation}
\begin{equation}
\label{eq:OTReDualCon2SolAlpha2}
\beta_{xy2}(k+1) = \beta_{xy2}(k) + \widehat{\eta} (w_{xy}(k+1) -w_{xy}^{(s)}(k+1)),
\end{equation}
where $\mathcal{U}_x:=\{ u_x,{\mathbf{w}_x^{(t)}} |  \gamma_{xy} - u_x - w_{xy}^{(t)} \leq 0, y\in\mathcal{Y}  \}$, $\mathcal{U}_y:=\{v_y,{\mathbf{w}_y^{(s)}} | \delta_{xy} - v_y + w_{xy}^{(s)} \leq 0, x\in\mathcal{X}  \}$. Note that (\ref{eq:OTReDualCon2SolBeta}) can be solved directly as $w_{xy}(k+1) = \frac{1}{2\widehat{\eta}}(\beta_{xy1}(k) - \beta_{xy2}(k))  + \frac{1}{2}(w_{xy}^{(t)}(k+1)+w_{xy}^{(s)}(k+1))$. As a result, iterations (\ref{eq:OTReDualCon2SolAlpha1}) and (\ref{eq:OTReDualCon2SolAlpha2}) can be written as $
\beta_{xy1}(k+1) = \frac{1}{2}(\beta_{xy1}(k)+\beta_{xy2}(k))+\frac{\widehat{\eta}}{2}(w_{xy}^{(t)}(k+1)-w_{xy}^{(s)}(k+1))$, and $
\beta_{xy2}(k+1) = \frac{1}{2}(\beta_{xy1}(k)+\beta_{xy2}(k))+\frac{\widehat{\eta}}{2}(w_{xy}^{(t)}(k+1)-w_{xy}^{(s)}(k+1))$, respectively. Thus, we have that $\beta_{xy1}(k) = \beta_{xy2}(k)$ holds for $k> 0$. As a result, after writing $\beta_{xy1}(k)$ and $\beta_{xy2}(k)$  as $\beta_{xy}(k)$, we have  (\ref{eq:OTReDualCon2Sol2Beta}) and  (\ref{eq:OTReDualCon2Sol2Alpha}). 

\section*{Appendix E: Proof of Proposition \ref{pro:OTReDualSolDual}}
We first derive the dual problem of (\ref{eq:OTReDualCon2Sol2UW}). Note that the Lagrange function of (\ref{eq:OTReDualCon2Sol2UW}) is given by:
\begin{equation}
\label{eq:xUWLagrange}
\begin{array}{l}
L_{x}( u_x,\mathbf{w}_x^{(t)},{\Lambda_x^{(t)}}) =   u_xp_x + \sum\limits_{y\in\mathcal{Y}} \beta_{xy}(k)w_{xy}^{(t)} \\ \ \   + \frac{\widehat{\eta}}{2} \sum\limits_{y\in\mathcal{Y}} \left(w_{xy}^{(t)} - w_{xy}(k) \right)^2 + \sum\limits_{y\in\mathcal{Y}} \lambda_{xy}^{(t)} \left( \gamma_{xy} - u_x - w_{xy}^{(t)}  \right),
\end{array}
\end{equation}
where $\lambda_{xy}^{(t)} \geq 0$ are the Lagrange multipliers. {By KKT conditions from Section 5.5.3 in \cite{boyd2004convex}, we have
$\frac{\partial L_x }{\partial u_x} = p_x - \sum\limits_{y\in\mathcal{Y}} \lambda_{xy}^{(t)} = 0$; $\frac{\partial L_x }{\partial w_{xy}^{(t)}} =\beta_{xy}(k) + \widehat{\eta} w_{xy}^{(t)} - \widehat{\eta} w_{xy}(k) - \lambda_{xy}^{(t)} = 0$.} Thus, we can achieve $ \sum\limits_{y\in\mathcal{Y}} \lambda_{xy}^{(t)} = p_x$, $   w_{xy}^{(t)}  = w_{xy}(k) + \frac{1}{\widehat{\eta}} (\lambda_{xy}^{(t)}  - \beta_{xy}(k))$, and $\lambda_{xy}^{(t)}\geq 0$. By plugging these equalities into (\ref{eq:xUWLagrange}), we have $\sum\limits_{y\in\mathcal{Y}} \lambda_{xy}^{(t)} \gamma_{xy} - \sum\limits_{y\in\mathcal{Y}} w_{xy}(k)  \left( \lambda_{xy}^{(t)} - \beta_{xy}(k) \right)  - \frac{1}{2\widehat{\eta}} \sum\limits_{y\in\mathcal{Y}} \left( \lambda_{xy}^{(t)}  - \beta_{xy}(k)\right)^2$. As a result, the dual problem of (\ref{eq:OTReDualCon2Sol2UW}) can be written as (\ref{eq:OTReDualSolDualUW}). Note that the terms $w_{xy}(k)\beta_{xy}(k)$ have been ignored as they are constant for $\lambda_{xy}^{(t)}$. After solving (\ref{eq:OTReDualSolDualUW}), we can find $w_{xy}^{(t)}(k+1)$ with
\begin{equation}
\label{eq:xDualw}
w_{xy}^{(t)}(k+1)   =  w_{xy}(k) + \frac{1}{\widehat{\eta}} (\lambda_{xy}^{(t)}(k+1)  - \beta_{xy}(k)).
\end{equation}

Following a similar method, we can prove that the dual problem of (\ref{eq:OTReDualCon2Sol2VW}) is (\ref{eq:OTReDualSolDualVW}). After solving (\ref{eq:OTReDualSolDualVW}), we can find $w_{xy}^{(s)}(k+1)$ with:
\begin{equation}
\label{eq:yDualw}
w_{xy}^{(s)}(k+1)   = w_{xy}(k) + \frac{1}{\widehat{\eta}} (\beta_{xy}(k)-\lambda_{xy}^{(s)}(k+1) ).
\end{equation}

By plugging (\ref{eq:xDualw}) and (\ref{eq:yDualw}) into (\ref{eq:OTReDualCon2Sol2Beta}) and (\ref{eq:OTReDualCon2Sol2Alpha}), we can achieve (\ref{eq:OTReDualSolDualAlpha}) and (\ref{eq:OTReDualSolDualBeta}), respectively.

\bibliographystyle{ieeetr}
\bibliography{DraftRZ.bib}

\end{document}